\documentclass[a4paper,11pt, reqno]{amsart}
\usepackage{etex}
\reserveinserts{28}
\usepackage[T1]{fontenc}
\usepackage{silence,lmodern}
\usepackage[utf8]{inputenc}
\usepackage[english]{babel}
\usepackage[stretch=10]{microtype}
\usepackage{mathtools}
\usepackage{amsthm}
\usepackage{amssymb}
\usepackage{xcolor,colortbl}
\usepackage{wrapfig}
\usepackage{rotating}

\usepackage{url}
\usepackage{enumerate}
\usepackage{bbm}
\usepackage{color,tikz,pgfplots}
\usepackage{pst-func}
\usepackage[thinlines]{easytable}
\usepackage{graphicx}
\usepackage{caption}
\usepackage[makeroom]{cancel}

\usepackage{nicefrac}

\usepackage{subcaption}
\usepackage{float}
\usepackage{csquotes}
\usepackage[left=3cm, top=2.5cm,bottom=2cm,right=3cm]{geometry}
\usepackage{bm}
\usepackage{makecell}
\newtheorem{theorem}{Theorem}
\newtheorem{question}{Question}
\newtheorem{lemma}{Lemma}
\newtheorem{proposition}{Proposition}
\newtheorem{corollary}{Corollary}

\newtheorem*{theoremWITHOUT}{Theorem}
\theoremstyle{definition}
\newtheorem{definition}{Definition}
\newtheorem{remark}{Remark}

\newtheorem*{exmp}{Example}

\numberwithin{equation}{section}

\newcommand{\bbW}{\textbf{W}}
\newcommand{\bbw}{\textbf{w}}
\newcommand{\bbZ}{\textbf{Z}}

\newcommand{\bbY}{\textbf{Y}}
\newcommand{\by}{\textbf{y}}
\newcommand{\bbS}{\textbf{S}}
\newcommand{\bbU}{\textbf{U}}
\newcommand{\cA}{\mathcal{A}}
\newcommand{\cL}{\mathcal{L}}

\newcommand{\cC}{\mathcal{C}}

\newcommand{\cT}{\mathcal{T}}

\newcommand{\cR}{\mathcal{R}}
\newcommand{\cK}{\mathcal{K}}
\newcommand{\dha}{d_{\mathbin{haus}}}
\newcommand{\cF}{\mathcal{F}}

\newcommand{\cI}{\mathcal{I}}
\newcommand{\cG}{\mathcal{G}}

\newcommand{\cM}{\mathcal{M}}
\newcommand{\cME}{\mathcal{M}^{\text{exch}}}

\newcommand{\bP}{\mathbb{P}}
\newcommand{\bE}{\mathbb{E}}

\newcommand{\bp}{\textbf{p}}
\newcommand{\bx}{\textbf{x}}

\newcommand{\bn}{\textbf{n}}

\newcommand{\dW}{d^*}

\newcommand{\bw}{\textbf{w}}
\newcommand{\bbX}{\textbf{X}}
\newcommand{\bbG}{\textbf{G}}
\newcommand{\bv}{\textbf{v}}

\newcommand{\bbJ}{\textbf{J}}
\newcommand{\bbV}{\textbf{V}}
\newcommand{\bmu}{\bm{\mu}}
\newcommand{\bbF}{\bm{\cF}}

\newcommand{\blfi}{\blacksquare}
\newcommand{\whfi}{\square}
\newcommand{\blgr}{\color{black!10!white}\blacksquare\color{black}}
\newcommand{\whgr}{\color{black!10!white}\square\color{black}}

\newcommand{\bbeta}{\bm{\eta}}

\newcommand{\bR}{\mathbb{R}}

\newcommand{\bS}{\mathbb{S}}

\newcommand{\bL}{\mathbb{L}}
\newcommand{\bN}{\mathbb{N}}

\newcommand{\as}{\overset{a.s.}{=}}
\newcommand{\subas}{\overset{a.s.}{\subseteq}}

\pgfplotsset{compat=1.6}

\DeclareMathOperator*{\largecap}{\mathbin{\scalebox{1.5}{\ensuremath{\bigcap}}}}
\DeclareMathOperator{\unif}{\mathbin{unif}}

\DeclareMathOperator{\RSS}{\mathtt{RSS}}
\DeclareMathOperator{\RSSB}{\partial_{\RSS}}
\DeclareMathOperator{\sample}{\mathtt{PosSam}}
\DeclareMathOperator{\spread}{\mathtt{IOS}}
\usepackage{hhline}

\DeclareMathOperator{\erase}{\mathtt{erase}}
\DeclareMathOperator{\ios}{\mathtt{ios}}
\DeclareMathOperator{\os}{\mathtt{os}}
\DeclareMathOperator{\set}{\mathtt{set}}
\DeclareMathOperator{\ps}{\mathtt{ps}}

\DeclareMathOperator{\erg}{\mathbin{erg}}

\DeclareMathOperator{\id}{\mathbin{id}}

\DeclareMathOperator{\law}{\cL}
\DeclareMathOperator{\GL}{\mathtt{GEWP}}
\DeclareMathOperator{\density}{\mathtt{ssd}}

\author[J. Gerstenberg]{Julian Gerstenberg}
\address{Julian Gerstenberg: Institut f\"ur Mathematische Stochastik, Leibniz Universit\"at Hannover, Welfengarten 1, 30167 Hannover, Germany}
\email{jgerst@stochastik.uni-hannover.de}


\begin{document}
	
	\title[General Erased-Word Processes]{General Erased-Word Processes:\\ Product-Type Filtrations, Ergodic Laws and Martin Boundaries}
	
	
	\begin{abstract}
		We study the dynamics of erasing randomly chosen letters from words by
		introducing a certain class of discrete-time stochastic processes, general erased-word
		processes (GEWPs), and investigating three closely related topics: Representation, Martin boundary and filtration theory. We use de Finetti's theorem
		and the random exchangeable linear order to obtain a de Finetti-type
		representation of GEWPs involving induced order statistics. Our studies expose 
		connections between exchangeability theory and certain poly-adic filtrations that
		can be found in other exchangeable random objects as well. We show that ergodic GEWPs generate backward filtrations of product-type and by that generalize a result by S.Laurent \cite{laurent}.
	\end{abstract}
	\maketitle
	\vspace{-0.5cm}
		
	\section{Introduction}
	Let $A$ be a measurable space, \emph{the alphabet}, and $\bw=(w_1,\dots,w_n)\in A^n$ a word of length $n$ over $A$. For $i\in[n]:=\{1,\dots,n\}$ we define $\erase(\bw,i)\in A^{n-1}$ to be the word obtained by erasing the $i$-th letter of $\bw$, that is $$\erase(\bw,i)=(w_1,\dots,w_{i-1},w_{i+1},\dots,w_n).$$ We use bold letters to indicate that a (random) variable represents a vector (word) or a sequence, e.g. $\bw=(w_1,\dots,w_n)$ or $\bbX=(X_j)_{j\geq 1}$. 
	\begin{definition}\label{def:gewp}
		A \emph{general erased-word process} (GEWP) over $A$ is a stochastic process $(\bbW,\bbeta)=(\bbW_n,\eta_n)_{n\geq 1}$ such that for each $n$
		\begin{enumerate}
			\item[(i)] $\bbW_n=(W_{n,1},\dots,W_{n,n})$ is a random word of length $n$ over $A$,
			\item[(ii)] $\eta_n$ is uniform on $[n]$ and independent of the $\sigma$-field $\cF_n:=\sigma(\bbW_k,\eta_{k+1}:k\geq n)$,
			\item[(iii)] $\bbW_n=\erase(\bbW_{n+1},\eta_{n+1})$ almost surely.
		\end{enumerate}
		We call $\bbW=(\bbW_n)_{n\geq 1}$ the \emph{word-process} and $\bbeta=(\eta_n)_{n\geq 1}$ the \emph{eraser-process} of $(\bbW,\bbeta)$.
	\end{definition}
	\begin{figure}
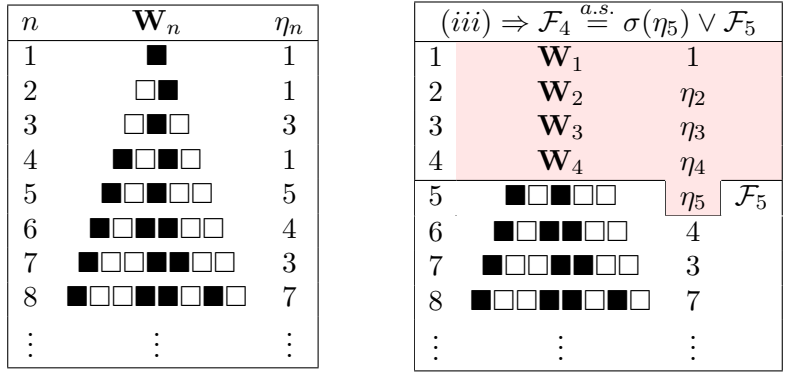

		\begin{tabular}{|ccc|}
			\hline
			$n$ & $\bbW_n$ & $\eta_n$\\
			\hline
			1& $\blacksquare$ & $1$\\
			
			2& $\square\blacksquare$ & $1$\\
			
			3& $\square\blacksquare\square$ & $3$\\
			
			4& $\blacksquare\square\blacksquare\square$ & $1$\\
			
			5& $\blacksquare\square\blacksquare\square\square$ & $5$\\
			
			6& $\blacksquare\square\blacksquare\blacksquare\square\square$ & $4$\\
			
			7& $\blacksquare\square\square\blacksquare\blacksquare\square\square$ & $3$\\
			
			8& $\blacksquare\square\square\blacksquare\blacksquare\square\blacksquare\square$ & $7$\\
			
			$\vdots$ & $\vdots$ & $\vdots$ \\
			
			\hline
		\end{tabular}
		\hspace{1cm}
		\begin{tabular}{|c c c c|} 
			\hline
			\multicolumn{4}{|c|}{$(iii)\Rightarrow\cF_4\as\sigma(\eta_5)\vee\cF_5$}\\
			\hline
			1 & \cellcolor{red!10}$\bbW_1$ &\cellcolor{red!10}$1$ &\cellcolor{red!10}\\
			2 & \cellcolor{red!10}$\bbW_2$ & \cellcolor{red!10}$\eta_2$&\cellcolor{red!10}\\
			3 & \cellcolor{red!10}$\bbW_3$  & \cellcolor{red!10}$\eta_3$&\cellcolor{red!10}\\
			4 & \cellcolor{red!10}$\bbW_4$  & \cellcolor{red!10}$\eta_4$&\cellcolor{red!10}\\
			\hhline{-->{\arrayrulecolor{red!10}}->{\arrayrulecolor{black}}-}
			\cellcolor{blue!0}5& \cellcolor{blue!0}$\blacksquare\square\blacksquare\square\square$ & \multicolumn{1}{|c|}{\cellcolor{red!10}$\eta_5$} &\cellcolor{blue!0}$\cF_5$\\
			\cline{3-3}\omit \vrule height.4pt\textcolor{blue!0}{\leaders\vrule\hfil}\vrule \cr 
			
			\cellcolor{blue!0}6& \cellcolor{blue!0}$\blacksquare\square\blacksquare\blacksquare\square\square$ & \cellcolor{blue!0}$4$&\cellcolor{blue!0}\\
			
			\cellcolor{blue!0}7& \cellcolor{blue!0}$\blacksquare\square\square\blacksquare\blacksquare\square\square$ & \cellcolor{blue!0}$3$&\cellcolor{blue!0}\\
			
			\cellcolor{blue!0}8& \cellcolor{blue!0}$\blacksquare\square\square\blacksquare\blacksquare\square\blacksquare\square$ & \cellcolor{blue!0}$7$&\cellcolor{blue!0}\\
			
			\cellcolor{blue!0}$\vdots$ & \cellcolor{blue!0}$\vdots$ & \cellcolor{blue!0}$\vdots$&\cellcolor{blue!0}\\
			
			\hline
		\end{tabular}
		\caption{Realization of a GEWP $(\bbW_n,\eta_n)_{n\geq 1}$ over $A=\{\blacksquare,\square\}$ and the information contained in 
			$\cF_5=\sigma(\bbW_k,\eta_{k+1}:k\geq 5)$.}
		\label{im}
	\end{figure}
	\newpage
	We study three closely related topics, more background is presented later:\\
	\textbf{1. Representation theory.}~What GEWPs exist? This topic is closely related to exchangeability theory and we use de~Finetti's theorem to obtain a de~Finetti-type representation result involving induced order statistics, see Theorem~\ref{thm:main1}. In this part of our studies we assume that the alphabet $A$ is a Borel space.\\
	\textbf{2. Martin boundary theory.}~ What is the topological behavior of GEWPs as ${n\rightarrow\infty}$? We present a homeomorphic description of a certain Martin boundary associated to GEWPs, see Theorem~\ref{thm:main2}. In this part we assume that the alphabet $A$ is a polish space. As a corollary we obtain a result located in the area of 'Limits of discrete structures', Corollary~\ref{cor:ssd}. The latter is closely related to a recent result by Choi~\&~Evans~\cite{evanschoi}.\\
	\textbf{3. Filtration theory.}~What is the informational behavior of GEWPs 'near' time $n=\infty$? This question is about the nature of backward filtrations generated by GEWPs, which are examples of so-called poly-adic filtrations: for each $n$ there exists a random variable uniformly distributed on a finite set, $\eta_n$, that is independent from $\cF_n$ and closes the 'informational gap' between $\cF_{n-1}$ and $\cF_n$, i.e. $\cF_{n-1}\as\sigma(\eta_n)\vee\cF_n$ (see Figure~\ref{im}). Given a poly-adic filtration, one is interested in the question if the filtration is generated by a sequence of independent RVs. We answer this question in Theorem~\ref{thm:main3}. Here we assume that $A$ is a Borel space.
	
	Besides the theorems, we think an essential contribution of our work lies in the presentation of the cross-connections between these three topics, as these connections can easily be translated to other erased-type processes, see Section~2.4. To our knowledge, in particular the connection between exchangeability and the theory of poly-adic filtrations has not get much attention yet. This connection is build upon folklore around the \emph{exchangeable linear order on $\bN$}, see Section~3.2. 
	
	We explain why we call the processes under consideration \emph{general} erased-word process: Laurent \cite{laurent} introduced so-called \emph{erased-word processes} (EWPs), where a stochastic process $(\bbW,\bbeta)=(\bbW_n,\eta_n)_{n\geq 1}$ is an EWP if (i)-(iii) from Definition~\ref{def:gewp} and the additional assumption
	\begin{enumerate}
		\item[(iv)] for each $n$ the letters $W_{n,1},\dots,W_{n,n}$ of $\bbW_n$ are iid
	\end{enumerate}
	hold. Because GEWPs are generalizations of EWPs, we use the additional \emph{general} in the name. We explain Laurent's results and their relations to our work in Section~2.3.
	
	Our paper is structured as follows:
	\vspace{-1.3cm}
	\renewcommand{\contentsname}{} 
	\tableofcontents
	
	\subsection{Notations}\label{sec:intro}
	We fix some notations and present basic knowledge as can be found in \cite{kallenberg}. Let $(S,\cR)$ be a measurable space. We usually do not name the $\sigma$-field $\cR$ explicitly, but for clarity we do so here. $\cM_1(S)$ denotes the set of probability measures on $S$, i.e. the set of $\sigma$-additive functions $P:\cR\rightarrow[0,1]$ with $P(S)=1$. We equip $\cM_1(S)$ with the $\sigma$-field generated by the projections $i_B:P\mapsto P(B)$ for $B\in\cR$. If $X$ is a $S$-valued random variable (RV) we denote the law of $X$ with $\law(X)\in\cM_1(S)$. If $\Xi$ is a $\cM_1(S)$-valued RV (a random probability measure) with law $\mu=\law(\Xi)\in\cM_1(\cM_1(S))$, the expectation of $\Xi$ (or disintegration of $\mu$) is given by the probability measure $\bE(\Xi)=\int Qd\mu(Q)\in\cM_1(S)$ defined by $B\in\cR\mapsto \bE(\Xi(B))=\int Q(B)d\mu(Q)$. Let $X$ be a $S$-valued RV defined on some background probability space $(\Omega,\cA,\bP)$ and let $\cF\subseteq \cA$ be a sub-$\sigma$-field. A regular conditional distribution of $X$ given $\cF$ is a $\cF$-measurable $\cM_1(S)$-valued RV, which we denote with $\cL(X|\cF)$, such that for each $B\in\cR$ it holds that $\cL(X|\cF)(B)=\bP(X\in B|\cF)=\bE(1_{\{X\in B\}}|\cF)$ almost surely. Regular conditional distributions do not exist under all circumstances, but they do if $S$ is a Borel space: A measurable space $S$  is called \textbf{Borel space} if there exists a Borel subset $E\subseteq [0,1]$ and a bijective map $f:S\rightarrow E$ such that $f, f^{-1}$ are measurable. If $S$ is a Borel space, then $\cL(X|\cF)$ exist, is almost surely unique and it holds that $\bE(\cL(X|\cF))=\cL(X)$. If $S$ is a Borel space, so is $\cM_1(S)$.\\
	If $\cF,\cG$ are two sub-$\sigma$-fields we write $\cF\subas\cG$ if for all $F\in\cF$ there is a $G\in\cG$ such that $\bP(F\Delta G)=0$ and we write $\cF\as\cG$ if both $\cF\subas\cG$ and $\cG\subas\cF$.
	If $S$ is a Borel space, $X$ is a $S$-valued RV and $Y$ is $S'$-valued RV then $\sigma(X)\subas\sigma(Y)$ if and only if there exists a measurable function $f:S'\rightarrow S$ with $X=f(Y)$ almost surely. Finally, we say that a $\sigma$-field $\cF$ is almost surely (a.s.) trivial if $\cF\as\{\emptyset,\Omega\}$. $\cF$ is a.s. trivial if and only if $\bP(F)\in\{0,1\}$ for all $F\in\cF$.
	
	\section{Main Results}
	
	\subsection{Representation Theory (Theorem 1)} We are interested in a description of the set of \emph{possible laws of GEWPs over a Borel space alphabet $A$}. A GEWP $(\bbW,\bbeta)=(\bbW_n,\eta_n)_{n\geq 1}$ over $A$ can be considered a random variable taking values in the Borel space ${\prod_{n\geq 1}A^n\times[n]}$. We introduce the set
	\begin{equation*}
		\cM(A):=\Big\{\law(\bbW,\bbeta): ~\text{$(\bbW,\bbeta)$ is a GEWP over $A$}\Big\},
	\end{equation*}
	i.e. $\cM(A)\subset\cM_1(\prod_nA^n\times[n])$. In Section~3 we show that GEWPs over $A$ are in one-to-one correspondence with \emph{jointly exchangeable pairs $(\bbY,L)$}, where $\bbY=(Y_j)_{j\geq 1}$ is an $A$-valued stochastic process and $L$ is random linear order on $\bN$. The random variable $Y_j\in A$ corresponds to the letter that gets erased by passing from $\bbW_j$ to $\bbW_{j-1}$ and $L$ is the unique linear order on $\bN$ such that $\eta_n$ equals the rank of $n$ in $\{1,2,\dots,n\}$ with respect to $L$ for each $n$. By identifying GEWPs with exchangeable random objects we find that the set $\cM(A)$ is a \emph{simplex}, i.e. a convex set in which every point has a unique representation as a mixture of extreme points. Exchangeability theory can be seen under an ergodic theory point-of-view and it is thus common to call laws that are extreme points \emph{ergodic} in such situations. We come back to the exchangeability point of view in Section~3. We define
	\begin{align*}
		\erg\cM(A)&:=\Big\{\law(\bbW,\bbeta): ~\text{$(\bbW,\bbeta)$ is an ergodic GEWP over $A$}\Big\}\\
				  &:=\text{extreme points of the convex set $\cM(A)$}.
	\end{align*}
	$\cM(A)$ being a simplex means that for any $P\in\cM(A)$ there exists a unique probability measure $\alpha\in\cM_1(\erg\cM(A))$ such that $P=\int_{\erg\cM(A)}Qd\alpha(Q)$ and this establishes a one-to-one correspondence between $\cM_1(\erg\cM(A))$ and $\cM(A)$. There are other ways of seeing that $\cM(A)$ is a simplex except from the connection to exchangeability: Dynkin \cite{dynkin} developed a theory of simplices in a measure theoretic framework that is applicable in our situation, we give details in Section~\ref{app:simpl}. Applying his results yields the following characterization of ergodicity:
	\begin{equation*}
		\text{$(\bbW,\bbeta)$ is ergodic if and only if $\cF_{\infty}:=\bigcap\nolimits_{n\geq 1}\cF_n$ is almost surely trivial.}
	\end{equation*} 
	Moreover, for any GEWP $(\bbW,\bbeta)$ over $A$ the conditional law $\law(\bbW,\bbeta|\cF_{\infty})$ is almost surely element of $\erg\cM(A)$ and the unique $\alpha\in\cM_1(\erg\cM(A))$ representing $\law(\bbW,\bbeta)$ as a mixture over extreme points is given by $\alpha=\law(\cL(\bbW,\bbeta|\cF_{\infty}))$. Theorem~\ref{thm:main1} below yields a unique description of $\erg\cM(A)$. 
	
	We need to introduce some notations from statistics, which play an important role throughout the paper: Let $\bS_n$ be the set of permutations of $[n]$. For real numbers $x_1,\dots,x_n$ with $x_i\neq x_j$ for all $i\neq j$ there exists a unique permutation $\pi\in\bS_n$ that arranges the $x$-values strictly increasingly, i.e. such that 
	$x_{\pi(1)}< \dots< x_{\pi(n)}$ holds. If $x_i=x_j$ for some $i\neq j$ there is no unique permutation arranging the values increasingly (non-strict), we choose $\pi$ to be the one that does it stably, i.e. $i<j\wedge x_{i}=x_{j}\Rightarrow\pi(i)<\pi(j)$. The \emph{permutation statistics} and \emph{order statistics} of $x_1,\dots,x_n$ are defined as 
	\begin{equation*}
		\ps(x_1,\dots,x_n):=\pi~~~\text{and}~~~\os(x_1,\dots,x_n):=(x_{1:n},\dots,x_{n:n}):=(x_{\pi(1)},\dots,x_{\pi(n)}).
	\end{equation*}
	In particular, $x_{1:n}\leq x_{2:n}\leq\cdots\leq x_{n:n}$ for all $x_1,\dots,x_n$ and $x_{1:n}<x_{2:n}<\dots<x_{n:n}$ holds iff $x_i\neq x_j$ for all $i\neq j$ iff $\ps(x_1,\dots,x_n)$ is the unique permutation arranging $x_1,\dots,x_n$ non-decreasingly. If $y_1,\dots,y_n$ are $A$-valued and $x_1,\dots,x_n$ are real valued, the \emph{induced order statistics of $y_1,\dots,y_n$ with respect to $x_1,\dots,x_n$} is defined as
	\begin{equation*}
		\ios(y_1,\dots,y_n,x_1,\dots,x_n):=(y_{\pi(1)},\dots,y_{\pi(n)})~~\text{with}~~\pi:=\ps(x_1,\dots,x_n).
	\end{equation*}
	Note that order statistics are real-valued vectors and induced order statistics are $A$-valued vectors, i.e. words over $A$. We direct the reader to \cite{bha} for an introduction to the statistical analysis of induced order statistics. 
	
	We now introduce the space used in Theorem~1 to parametrize $\erg\cM(A)$ bijectively:
	\begin{align*}
		\cC(A):&=\Big\{\rho\in\cM_1(A\times\bR):\rho(A\times[0,t])=t~\text{for all}~t\in[0,1]\Big\}\\
			   &=\Big\{\law(Y,U):~\text{$(Y,U)$ is a $A\times\bR$-valued RV with $\law(U)=\unif[0,1]$}\Big\}.
	\end{align*}
	\begin{theorem}\label{thm:main1}
		Let $\rho\in\cC(A)$ and $(Y_j,U_j)_{j\geq 1}$ iid~$\sim\rho$. For each $n$ define
		\begin{equation*}
			\bbW_n:=\ios(Y_1,\dots,Y_n,U_1,\dots,U_n)~~~\text{and}~~~\eta_n:=\#\{k\in[n]:U_k\leq U_n\}.
		\end{equation*}
		Then $(\bbW,\bbeta)=(\bbW_n,\eta_n)_{n\geq 1}$ is an ergodic GEWP. Writing $\GL(\rho):=\law(\bbW,\bbeta)\in\erg\cM(A)$ the map
		\begin{equation*}
			\cC(A)\rightarrow\erg\cM(A),~~\rho\mapsto\GL(\rho)
		\end{equation*} 
		is a bijection.
	\end{theorem}
	The simplex property yields that a $\prod_nA^n\times[n]$-valued RV $(\bbW,\bbeta)$ is a GEWP iff there exists a probability measure $\alpha\in\cM_1(\cC(A))$ such that $\law(\bbW,\bbeta)=\int_{\cC(A)}\GL(\rho)d\alpha(\rho)$. In Section~3 we see that this can be strengthened in the following sense: For every GEWP $(\bbW,\bbeta)$ there exists a $\cC(A)$-valued random variable $\Xi$, defined on the same underlying probability space as $(\bbW,\bbeta)$, such that $\law(\bbW,\bbeta|\Xi)=\GL(\Xi)$ almost surely. Moreover, $\Xi$ is a.s. unique. In exchangeability theory such random measures are called \emph{random directing measures}.
	
	Theorem \ref{thm:main1} tells us that the distributions of words $\bbW_n$ that appear in the word-chain of ergodic GEWPs are always laws of induced order statistics of $Y_1,\dots,Y_n$ with respect to $U_1,\dots,U_n$ where $(Y_i,U_i)$ are iid with $U_i\sim\unif[0,1]$. We use the following notation: For $\rho\in\cM_1(A\times\bR), n\geq 1$ and $(Y_1,X_1),\dots,(Y_n,X_n)$ iid $\sim\rho$ we define
	\begin{equation}\label{eq:rios}
		\spread(\rho,n):=\law(\ios(Y_1,\dots,Y_n,X_1,\dots,X_n))~\in\cM_1(A^n).
	\end{equation}
	We use '$\spread$' to indicate that we talk about the law of an induced order statistics of certain random variables, the letters are capital to distinguish it from the function '$\ios$'.
	
	We now consider the multi-step co-transition behavior of the word-process in GEWPs. If one repeatedly erases uniformly chosen letters in a word $\bbw=(w_1,\dots,w_n)$ of length $n$ until one ends up with a word of length $k\leq n$, the resulting random word has the same distribution as if one chooses one of the $\binom{n}{k}$ subsequences in $\bw$ uniformly at random. The distribution of a random subsequence ($\RSS$) of length $k$ extracted from $\bbw$ is given by
	\begin{equation*}
		\RSS(\bbw,k):=\frac{1}{\binom{n}{k}}\sum\limits_{1\leq j_1<\dots<j_k\leq n}\delta_{(w_{j_1},\dots,w_{j_k})}~~\in\cM_1(A^k),
	\end{equation*}
	where $\delta_x$ is the Dirac measure at $x$. The probability measures $\RSS(\bbw,k), k\geq 1, \bbw\in\cup_{n\geq k}A^n$ can be seen as a system of \emph{co-transition kernels on the graded family $A^k, k\geq 1$} fulfilling the \emph{Chapman-Kolmogorov equations}: Given a word $\bbw$ of length $n$ and $k\leq m\leq n$, choosing a uniform subsequence of length $m\leq n$ from $\bbw$ and then choosing a uniform subsequence of length $k\leq m$ from that subsequence yields the same final distribution as choosing a uniform subsequence of length $k\leq n$ from $\bbw$ in the first place. As a formula:
	\begin{equation}\label{eq:cotrans}
		\RSS(\bbw,k)=\int_{A^m}\RSS(\bv,k)~d\RSS(\bbw,m)(\bv).
	\end{equation}
	Given a system of co-transition kernels satisfying (\ref{eq:cotrans}) one is interested in the behavior of Markov chains sharing this co-transition dynamic: 
	\begin{definition}
		A stochastic process $\bbW=(\bbW_n)_{n\geq 1}$ with $\bbW_n\in A^n$ for each $n$ is called a $\RSS$-chain if 
		\begin{equation*}
			\law\big(\bbW_k|\sigma(\bbW_m; m\geq n)\big)=\RSS(\bbW_n,k)~\text{almost surely for all $k\leq n$}.
		\end{equation*}
		We define 
		\begin{equation*}
			\cM'(A):=\Big\{\law(\bbW): ~\text{$\bbW$ is a $\RSS$-chain over $A$}\Big\}.
		\end{equation*}
	\end{definition}
	$\RSS$-chains over $A$ are Markov chains $(\bbW_1,\bbW_2,\dots)$ with $\bbW_n\in A^n$ that evolve backwards in time by erasing uniformly chosen letters, hence the multi-step co-transitions are given by choosing uniform subsequences. 
	\vspace{-0.2cm}
	\begin{figure}[H]
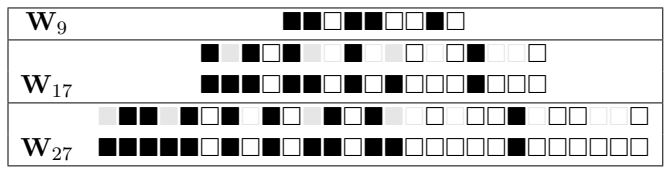

		\scalebox{0.9}{
		\begin{tabular}{|cc|}
			\hline
			$\bbW_9$   & $\blfi\blfi\whfi\blfi\blfi\whfi\whfi\blfi\whfi$\\
			\hline
					   & $\blfi\blgr\blfi\whfi\blfi\blgr\whgr\blfi\whgr\blgr\whfi\whgr\whfi\blfi\whgr\whgr\whfi$ \\
			$\bbW_{17}$& $\blfi\blfi\blfi\whfi\blfi\blfi\whfi\blfi\whfi\blfi\whfi\whfi\whfi\blfi\whfi\whfi\whfi$ \\
			\hline
			           & $\blgr\blfi\blfi\blgr\blfi\whfi\blfi\whgr\blfi\whfi\blgr\blfi\whfi\blfi\blgr\whgr\whfi\whgr\whfi\whfi\blfi\whgr\whfi\whfi\whgr\whgr\whfi$\\
			$\bbW_{27}$& $\blfi\blfi\blfi\blfi\blfi\whfi\blfi\whfi\blfi\whfi\blfi\blfi\whfi\blfi\blfi\whfi\whfi\whfi\whfi\whfi\blfi\whfi\whfi\whfi\whfi\whfi\whfi$ \\
			\hline
		\end{tabular}}
		\caption{$\RSS$-chain over $A=\{\blacksquare,\square\}$, $\bbW_9$ has distribution $\RSS(\bbW_{27},9)$.}
		\label{fig:rss}
	\end{figure}
	\vspace{-0.4cm}
	The connection of $\RSS$-chains to GEWPs is obvious: If$(\bbW,\bbeta)$ is a GEWP, then the word-process $\bbW$ is a $\RSS$-chain. This is due to the fact that the eraser $\eta_n$ used to go from $\bbW_n$ to $\bbW_{n-1}$ is independent from $\sigma(\bbW_{m}; m\geq n)$ and uniformly distributed on $[n]$. The Daniell-Kolmogorov existence theorem (\cite{kallenberg}, Theorem 5.14) yields the reverse statement: For any $\RSS$-chain $\bbW$ there exists a GEWP $(\bbW^*,\bbeta)$ such that $\cL(\bbW)=\cL(\bbW^*)$. Note that one needs a Borel space assumptions to apply this existence theorem. The law of a GEWP $(\bbW,\bbeta)$ is clearly determined by $\law(\bbW)$, hence the map
	\begin{equation*}
		\cM(A)\rightarrow\cM'(A),~\law(\bbW,\bbeta)\mapsto\law(\bbW)
	\end{equation*}
	is an affine measurable bijection (affine isomorphism). In particular, it maps extreme points (ergodic laws) to extreme points. One obtains a different way to see that $\cM(A)$ is a simplex: One can use the theory of Dynkin \cite{dynkin} to show that $\cM'(A)$ is a simplex (see Section \ref{app:simpl}) and then use that simplex property is preserved by measurable affine bijections (end of p.2 in \cite{dynkin}). As a consequence, a $\RSS$-chain is ergodic (:= its law is extreme point of $\cM'(A)$) iff the terminal $\sigma$-field generated by $\bbW$ is a.s. trivial. Moreover, for any GEWP $(\bbW,\bbeta)$ the terminal $\sigma$-field generated by the word-process $\bbW$ is a.s. equal to $\cap_n\cF_n$. Theorem~\ref{thm:main1} directly yields a representation result for $\RSS$-chains, one just forgets about $\bbeta$.
	
	Since $\RSS$-chains are Markov chains, the law of a $\RSS$-chain $\bbW=(\bbW_k)_{k\geq 1}$ is determined by the sequence of marginal laws $(\law(\bbW_k))_{k\geq 1}$. By Daniell-Kolmogorov existence theorem a sequence of laws $\bmu=(\mu_k)_{k\geq 1}$ with $\mu_k\in\cM_1(A^k)$ can appear as a sequence of marginal laws of a $\RSS$-chain if and only if 
	\begin{equation}\label{eq:maringallaws}
		\mu_k=\int_{A^n}\RSS(\bw,k)d\mu_n(\bw)~~\text{for all}~k\leq n.
	\end{equation}
	Let $\cM''(A)$ be the set of sequences $\bmu=(\mu_k)_{k\geq 1}$ satisfying (\ref{eq:maringallaws}). It is obvious that the map $\cM'(A)\rightarrow\cM''(A), \law(\bbW)\mapsto (\law(\bbW_k))_{k\geq 1}$ is a affine isomorphism, hence the convex subset $\cM''(A)\subset\prod_k\cM_1(A^k)$ is also a simplex. As a consequence of Theorem~\ref{thm:main1}, a sequence of laws $\bmu=(\mu_k)_{k\geq 1}$ satisfying (\ref{eq:maringallaws}) is an extreme point of $\cM''(A)$ if and only if there exists some $\rho\in\cC(A)$ such that 
	\begin{equation*}
		\mu_k=\spread(\rho,k)~~\text{for each}~k,
	\end{equation*}
	see (\ref{eq:rios}) for the definition of $\spread$ (random induced order statistics). We obtain
	\begin{corollary}\label{thm:main1alt}
		Let $A$ be a Borel space. A sequence of probability measures $\bmu=(\mu_k)_{k\geq 1}$ with $\mu_k\in\cM_1(A^k)$ for each $k$ satisfies (\ref{eq:maringallaws}) if and only if there exists a probability measure $\alpha\in\cM_1(\cC(A))$ such that 
		\begin{equation}\label{eq:id}
			\mu_k=\int_{\cC(A)}\spread(\rho,k)d\alpha(\rho)~~~\text{for all $k$.}
		\end{equation}
		(\ref{eq:id}) establishes a one-to-one correspondence between $\cM_1(\cC(A))$ and sequences $\bmu$ satisfying~(\ref{eq:maringallaws}).
	\end{corollary}
	

	\subsection{Martin Boundary Theory (Theorem 2)}
	Martin Boundary theory is a topological topic, hence we make a topological assumptions on the alphabet $A$. A natural assumption that is downwards compatible to the Borel space setting is to assume that $A$ is a polish space: A topological space $S$ is called polish space if there exists a metric $d$ on $S$ that generates the topology on $S$ and makes $(S,d)$ a complete separable metric space.
	The Borel $\sigma$-field on a polish space $S$ makes it a Borel space. If $S$ is a polish space, we equip $\cM_1(S)$ with the topology of weak convergence, i.e. the topology generated by the maps $i_f:P\mapsto \int_SfdP$ for $f:S\rightarrow \bR$ bounded continuous, which makes $\cM_1(S)$ a polish space again. Compatible metrics on $\cM_1(S)$ are given by Wasserstein distances based on compatible metrics on $S$, we work with this in Section~4. Note that 
	the Borel-$\sigma$-field on $\cM_1(S)$ coincides with the $\sigma$-field generated by $P\mapsto P(B), B\subset S$ measurable. In this section we assume that the alphabet $A$ is a polish space. We use that countable products of polish spaces are polish again and consider polish spaces of the form $A^k$ and $A\times\bR$.
	
	The version of Martin boundary theory we refer to is most usually studied in the context of (highly) transient countable state space Markov chains, where 'highly' transient means that for every state there is just one point in time in which the chain can attain this value with positive probability. Note that $\RSS$-chains over countable alphabets are of this form. We refer the reader to \cite{vershik, egw, ew, ge} for introductory literature.
	
	We note that it is not required to know anything about Martin boundary theory to understand the following definitions and theorems as we present the material in a self contained way. $|\bbw|\in\bN$ denotes the length of a word.
	
	\begin{definition}\label{def:rssconv} \begin{enumerate}
			\item[(i)] A sequence $(\bbw_n)_{n\geq 1}$ of words over $A$ is called \emph{$\RSS$-convergent}, if $|\bbw_n|\rightarrow\infty$ and for each $k$ the $\cM_1(A^k)$-valued sequence $(\RSS(\bbw_n,k))_{n\geq 1}$ converges weakly as $n\rightarrow\infty$. 
			\item[(ii)] The \emph{limit} of a $\RSS$-convergent sequence $(\bbw_n)_{n\geq 1}$ is given by $\bmu=(\mu_k)_{k\geq 1}$ with $\mu_k=\lim\limits_{n\rightarrow\infty}\RSS(\bbw_n,k)$ for each $k$.
			\item[(iii)] The \emph{Martin boundary of $\cup_{k\geq 1}A^k$ with respect to $\RSS$} is defined as
			$$\RSSB A:=\Big\{\bmu\in\prod\nolimits_{k\geq 1}\cM_1(A^k): \bmu~\text{is the limit of some $\RSS$-convergent sequence}\Big\}.$$
			We consider $\RSSB A$ to be a topological subspace of the polish product space $\prod_k\cM_1(A^k)$. 
		\end{enumerate}
	\end{definition}

	The following proposition lists some relations between $\RSS$-convergence, Martin boundary and $\RSS$-chains (and hence GEWPs). These relations are not special to our concrete situation, but follow from rather general theory about Markov chains with given co-transitions, see Remark~\ref{rem:gene} below. A proof of Proposition~\ref{prop:key} can be found in Section~\ref{app:prop}.
	
	\begin{proposition}\label{prop:key}
		\begin{enumerate}
			\item[(i)] Every $\RSS$-chain is almost surely $\RSS$-convergent.
			\item[(ii)] The a.s. limit of a $\RSS$-chains $\bbW$ is given by $(\law(\bbW_k|\cT))_{k\geq 1}$, where $\cT$ is the terminal $\sigma$-field generated by $\bbW$, i.e. $\cT=\cap_{n\geq 1}\sigma(\bbW_m:m\geq n)$.
			\item[(iii)] A $\RSS$-chain is ergodic iff its limit is almost surely constant.
			\item[(iv)] For every $\bmu\in\RSSB A$ there exists a $\RSS$-chain $\bbW$ such that $\law(\bbW_k)=\mu_k$ for all $k$.
			\item[(v)] For every ergodic $\RSS$-chain $\bbW$ exists $\bmu\in\RSSB A$ such that $\law(\bbW_k)=\mu_k$ for all $k$.
		\end{enumerate}
	\end{proposition}

	\begin{remark}\label{rem:gene}
		Proposition~\ref{prop:key} is not special to our concrete situation in the following sense: Let $S_k, k\geq 1$ be a sequence of polish spaces (e.g. $S_k=A^k$) and for all $k\leq n$ and $x\in S_n$ let $Q(x,k)\in\cM_1(A^k)$ be such that 
		$x\in A^n\mapsto Q(x,k)$ is measurable and $Q(x,k)=\int_{A^m}Q(y,k)dQ(x,m)(y)$ for each $k\leq m\leq n$ (e.g. $Q(x,k)=\RSS(\bbw,k)$). One can then introduce $Q$-chains, $Q$-convergence, limits and  Martin boundary analogously to the case of $Q=\RSS$ and 
		obtain all the statements form Proposition~\ref{prop:key}, possibly expect from~(iv), where a continuity assumption is needed. There are examples in which \emph{no $Q$-chains and no $Q$-convergent sequences exist}. A sufficient condition for existence is that the spaces $S_k, k\geq 1$ are compact metric spaces. See \cite{vershik} for general theory in the case where all $S_k$ are finite discrete spaces.
	\end{remark}

	Theorem~\ref{thm:main1} together with Proposition~\ref{prop:key} (v) yields that for every $\rho\in\cC(A)$ one can find a sequence $(\bbw_n)_{n\geq 1}$ of words over $A$ with $|\bbw_n|\rightarrow\infty$ such that
	\begin{equation}\label{eq:spreadrho}
		\spread(\rho,k)=\lim\limits_{n\rightarrow\infty}\RSS(\bbw_n,k)~~\text{for each $k$},
	\end{equation}
	where the convergence is weakly in $\cM_1(A^k)$. In particular, for any $\rho\in\cC(A)$ the sequence $(\spread(\rho,k))_{k\geq 1}$ is element of the Martin boundary $\RSSB A$. Proposition~\ref{prop:key}~(iv) yields that one can identify $\RSSB A$ with 
	\begin{equation*}
		\cM(\RSSB A):=\Big\{\law(\bbW,\bbeta)\in\cM(A):~\exists \bmu\in\RSSB A~\text{s.t.}~\law(\bbW_k)=\mu_k~\text{for all k}\Big\}
	\end{equation*}
	and (v) yields $\erg\cM(A)\subseteq\cM(\RSSB A)$ and hence the inclusion chain
	\begin{equation}\label{eq:incl}
		\erg\cM(A)\subseteq \cM(\RSSB A)\subseteq\cM(A).
	\end{equation}
	General Martin boundary theory can not tell whether these inclusions are strict or not, there are known cases for all four possible combinations. The case $\erg\cM(A)=\cM(A)$ can easily be detected: a simplex coincides with its extreme points iff the simplex consists of one point. In our situation this is the case iff $\#A=1$. In many concrete cases there is equality in the first inclusion. Our main theorem in this section, Theorem~\ref{thm:main2}, shows that this is also the case here: $\erg\cM(A)=\cM(\RSSB A)$ holds for any polish space $A$. In terms of marginal laws this means that limits of $\RSS$-convergent sequences are precisely given by sequences of the form $(\spread(\rho,k))_{k\geq 1}, \rho\in\cC(A)$.
	
	For $\bw=(w_1,\dots,w_n)\in A^n$ we define
	\begin{equation*}
		\rho_{\bw}:=\frac{1}{n}\sum_{j=1}^n\delta_{(w_j,\frac{j}{n})}~\in\cM_1(A\times\bR),
	\end{equation*}
	i.e. $\rho_{\bw}$ is the law of $(w_{J},\frac{J}{n})$, where $J\sim\unif[n]$. The map $\bw\mapsto \rho_{\bw}$ is an embedding of the space of all finite words over $A$  into the polish space $\cM_1(A\times\bR)$. We also need a \emph{mixed} version of $\rho_{\bw}$. For $\mu\in\cM_1(A^n)$ we define 
	\begin{equation*}
		\sample(\mu):=\int_{A^n}\rho_{\bv}d\mu(\bv)~\in\cM_1(A\times\bR),
	\end{equation*}
	i.e. $\sample(\mu)$ is the law of $(V_J,\frac{J}{n})$, where $\bbV=(V_1,\dots,V_n)$ has law $\mu$ and is independent from $J\sim\unif[n]$. In particular, $\sample(\delta_{\bbw})=\rho_{\bw}$. The abbreviation '$\sample$' stands for 'position sample'; $\sample(\mu)$ is the law obtained by first picking a random word with law $\mu$ and then sampling a letter from that word and remember also the (relative) position of the letter in the word.
	\begin{theorem}\label{thm:main2}
		Let $(\bw_n)_{n\geq 1}$ be a sequence of $A$-valued words with $|\bbw_n|\rightarrow\infty$ as $n\rightarrow\infty$. 
		The following two statements are equivalent:
		\begin{enumerate}
			\item[(i)] $(\bw_n)_{n\geq 1}$ is $\RSS$-convergent towards some limit $\bmu=(\mu_k)_{k\geq 1}\in\RSSB A$.
			\item[(ii)] $(\rho_{\bw_n})_{n\ge 1}$ converges in $\cM_1(A\times\bR)$ towards some $\rho\in\cC(A)$.
		\end{enumerate}
		If (i) and (ii) hold, one has 
		\begin{equation}\label{eq:onetoone}
			\rho = \lim_{k\rightarrow\infty}\sample(\mu_k)~~~\text{and}~~~\mu_k = \spread(\rho,k)~\text{for each $k$}
		\end{equation}
		and this establishes a homeomorphism between $\RSSB A$ and $\cC(A)$. 
	\end{theorem}

	The proof of Theorem~\ref{thm:main2} is presented in Section~4, we use Wasserstein distances and apply coupling methods to obtain bounds and show convergence.\\
	
	For the rest of Section~2.2 we discuss corollaries of Theorem~\ref{thm:main2}.
	
	\begin{corollary}
		$\erg\cM(A)=\cM(\RSSB A)$.
	\end{corollary}
	\begin{proof}
		Because of (\ref{eq:incl}) we only need to show $\cM(\RSSB A)\subset \erg\cM(A)$. Let $\law(\bbW,\bbeta)\in\cM(\RSSB A)$ via $\bmu\in\RSSB A$, i.e. $\law(\bbW_k)=\mu_k$ for all $k$. By Theorem~\ref{thm:main2} there is a $\rho\in\cC(A)$ such that $\spread(\rho,k)=\mu_k=\law(\bbW_k)$ for each $k$. By Corollary~\ref{thm:main1alt} $(\mu_k)_{k\geq 1}$ is an extreme point of $\cM''(A)$ and hence $\law(\bbW,\bbeta)\in\erg\cM(A)$. 
	\end{proof}

	\begin{corollary}
		If $(\bbW,\bbeta)$ is a GEWP over the polish alphabet $A$, then $\rho_{\bbW_n}$ converges almost surely weakly towards its random directing measure $\Xi$, i.e. $\Xi$ is a $\cC(A)$-valued RV with $\law(\bbW,\bbeta|\Xi)=\GL(\Xi)$ almost surely.
	\end{corollary}
	\begin{proof}
		By Proposition \ref{prop:key} $\bbW$ is almost surely $\RSS$-convergence with limit $(\cL(\bbW_k|\cT))_{k\geq 1}$, where $\cT$ is the terminal $\sigma$-field generated by the word-process $\bbW$. By Theorem~\ref{thm:main2} the a.s. $\RSS$-convergence of $\bbW$ implies the a.s. convergence of $\rho_{\bbW_n}$ towards some random probability measure $\Xi$ with $\bP(\Xi\in\cC(A))=1$. The theorem also implies $\cL(\bbW_k|\cT)=\spread(\Xi,k)$ almost surely for all $k$, which yields $\law(\bbW,\bbeta|\Xi)=\GL(\Xi)$ almost surely.
	\end{proof}

	Now we consider the case where $A$ is a finite alphabet (discrete topology) and explain how Theorem~\ref{thm:main2} contributes to the area of \emph{limits of discrete structures} in this case. In particular, we explain the connection of our work to a closely related paper by Choi~\&~Evans~\cite{evanschoi}. For $k\leq n, \bv=(v_1,\dots,v_k)\in A^k$ and $\bw=(w_1,\dots,w_n)\in A^n$ we consider
	\begin{equation*}
		\binom{\bw}{\bv}:=\#\Big\{1\leq j_1<j_2<\dots<j_k\leq n:~w_{j_i}=v_i~\text{for all}~i\in[k]\Big\},
	\end{equation*}
	i.e. $\binom{\bw}{\bv}$ counts how often the shorter word $\bv$ is embedded as subsequence in the longer word $\bw$. For example, if $A=\{a,b\}, \bv=aab$ and $\bw=abaab$ it holds that $\binom{\bw}{\bv}=3$, the embeddings of $\bv$ in $\bw$ are given by $\color{red}a\color{black}b\color{red}a\color{black}a\color{red}b\color{black}, \color{red}a\color{black}ba\color{red}ab\color{black}, ab\color{red}aab\color{black}$. A word of length $n$ has $\binom{n}{k}$ subsequences of length $k$, we introduce the \emph{subsequence density of $\bv$ in $\bw$} as
	\begin{equation*}
		\density(\bv,\bw):=\frac{\binom{\bw}{\bv}}{\binom{|\bw|}{|\bv|}}.
	\end{equation*}
	For example, $\density(aab,abaab)=3/\binom{5}{3}=0.3$. For $k=|\bv|\leq |\bw|=n$ the value $0\leq\density(\bv,\bw)\leq 1$ is the probability to pick $\bv$ when a subsequence of length $k$ is chosen uniformly from $\bw$. 
	
	\begin{corollary}[Subsequence density convergence]\label{cor:ssd}
		Let $A$ be finite and $(\bbw_n)_{n\geq 1}$ be a sequence of words over $A$ with $|\bbw_n|\rightarrow\infty$. Then the following two statements are equivalent:
		\begin{enumerate}
			\item[(i)] $\density(\bv,\bw_n)$ converges as $n\rightarrow\infty$ for each word $\bv$ over $A$.
			\item[(ii)] $\rho_{\bbw_n}$ converges weakly as $n\rightarrow\infty$ towards some $\rho\in\cC(A)$.
		\end{enumerate}
		If (i), (ii) hold, then $\lim_{n\rightarrow\infty}\density(\bv,\bw_n)=\spread(\rho,k)(\{\bv\})$ for each $k\geq 1, \bv\in A^k$.
	\end{corollary}
	\begin{proof}
		For each $k$ $A^k$ is a finite discrete space and hence $\cM_1(A^k)$ is in one-to-one correspondence with probability vectors and the topology on $\cM_1(A^k)$ is euclidean. In particular, for $\nu_n,\nu\in\cM_1(A^k)$ one has that $\nu_n$ converges weakly to $\nu$ if and only if $\nu_n(\{\bv\})\rightarrow\nu(\{\bv\})$ for each $\bv\in A^k$. Now let $\bbw=(w_1,\dots,w_n)\in A^n$ and consider $k\leq n$ and $\bv=(v_1,\dots,v_k)$. Since $\RSS(\bbw,k)(\{\bv\})=\density(\bv,\bbw)$ the result follows from Theorem~\ref{thm:main2}.
	\end{proof}
	\begin{remark}
		Corollary~\ref{cor:ssd} suggests that it is reasonable to call distributions of the form $\RSS(\bbw,k), \bbw\in A^n, n\geq k\geq 1$ \emph{ordered multivariate hypergeometric distributions} and distributions of the form 
		$\spread(\rho,k), \rho\in\cC(A) ,k\geq 1, $ \emph{ordered multinomial distributions}: Let $A=[m]$ be a finite set of colors. Consider an urn with $n$ balls of which $n_i$ are of color $i\in A$. Now pick $k\leq n=n_1+\dots+n_m$ balls uniformly without replacement. The resulting counting vector of drawn colors has (by definition) the multivariate hypergeometric distribution. Hence multivariate hypergeometric distributions are parametrized by pairs $(\bn,k)$ where $\bn=(n_1,\dots,n_m)$ describes the urn occupation and $k\leq n_1+\dots+n_m$ describes the number of drawn balls. The multinomial distributions can be obtained by considering fixed $k$ and sequence of urn occupations with $n_1+\dots+n_m=:n\rightarrow\infty$ and $n_i/n\rightarrow p_i$. Multinomial distributions are parametrized by pairs $(\bp,k)$ where $\bp=(p_1,\dots,p_m)$ is a probability vector and $k\geq 1$. Now consider the following ordered version of the experiment: the $n$ balls of which $n_i$ are of color $i\in A$ are no longer in an urn \emph{but lined up on a table}. The way they are lined up can be described by a word $\bw=(w_1,\dots,w_n)\in A^n$ with $n_i=\#\{k\in[n]:w_k=i\}$, where $w_i$ is the color of the $i$-th ball from the left. Now draw $k\leq n$ of the balls without replacement and \emph{keep the order of the drawn balls}, i.e. pick a uniform subsequence of length $k$ from $\bbw$. The resulting word of length $k$ has distribution $\RSS(\bw,k)$, hence the distribution family $\RSS(\bbw,k)$ parametrized by pairs $(\bbw,k)$ with $k\leq |\bw|$ can be considered the \emph{ordered} version of multivariate hypergeometric distribution. Considering limiting cases, i.e. $k$ fixed and $n\rightarrow\infty$, yields the ordered version of multinomial distributions. Corollary~\ref{cor:ssd} tells us that these are given by $\spread(\rho,k)$, hence are parametrized by pairs $(\rho,k)$ with $\rho\in\cC(A)$ and $k\geq 1$. The unordered version can be obtained from the ordered version by 'forgetting' about the order, i.e. by going from words $\bbw=(w_1,\dots,w_n)$ to counting vectors $\bn=(n_1,\dots,n_m), n_i=\#\{j:w_j=i\}$ and by going from probability measures $\rho\in\cC(A)$ to the first marginal $\bp=(p_1,\dots,p_m), p_i=\rho(\{i\}\times\bR)$.
	\end{remark}

		Choi \& Evans \cite{evanschoi} investigated the following situation: Given a finite alphabet $A=[m]$ they considered only such words $\bbw$ in which every letter from $A$ occurs evenly often; in particular they only considered words of length $n\cdot m$ for some $n$. Given such a word $\bbw\in A^{nm}$, they conditioned $\RSS(\bbw,km)$ on the event that in the resulting word every of the $m$ different letters occurs $k$-times. The authors introduced Martin boundaries analogously and obtained a representation result that is very similar to ours: The Martin boundary in their situation is in one-to-one correspondence with the set $\cC_{\text{even}}(A)=\{\rho\in\cC(A): \rho(\{i\}\times[0,1])=1/m~\text{for every $i\in A$}\}$. The method they used to obtain this result is very similar to our approach, i.e. exploring the connection of $\RSS$-chains and exchangeability first. A similar strategy to obtain descriptions of Martin boundaries was used before, see \cite{ew} and \cite{egw}.
	\begin{remark}
		One can introduce \emph{(ordered) embedding densities} for all kind of combinatorial structures, define a notation of convergence similar to Corollary~\ref{cor:ssd} (i) and ask for a nice description of the occurring limit density functions. \cite{hkmrs} considered permutations and identified limits with $2$-dimensional copulas, \cite{egw} studied ordered binary trees and \cite{ge} generalized this to non-binary ordered trees and beyond. In all situations there is a one-to-one correspondence between limits of convergent combinatorial structures and certain ergodic exchangeable laws involving joinings with exchangeable linear order (see Section~3). The case of permutation limits corresponds to the case of jointly exchangeable pairs of linear orders $(L,L')$, see Remark~\ref{rem:pairs} for more details. A general theory about these types of relations has been developed in the authors PhD thesis \cite{gdiss}. There are also 'unordered' versions of embedding density convergences closely connected to exchangeability theory, see \cite{austin, diaconisjanson} for the connection of graph limits and exchangeable random graphs. 
	\end{remark}
	Finally, we consider $A=\{0,1\}$. Here it is possible to give a nice graphical characterization of convergence of embedding densities. We mention this, because we later obtain a nice way of explaining graphically why certain constructions work (Filtration theory, Section~5). Let $\bw=(w_1,\dots,w_n)\in \{0,1\}^n$ and define the \emph{set} 
	\begin{equation}\label{eq:setw}
		\set(\bw):=\Big\{\Big(\frac{w_1+\dots+w_i}{n},\frac{i-(w_1+\dots+w_i)}{n}\Big):i\in\{0,1,\dots,n\}\Big\}.
	\end{equation}
	$\set(\bw)$ is a finite subset of the square $[0,1]^2$, hence compact. Let $\cK([0,1]^2)$ be the set of all non-empty compact subsets of $[0,1]^2$. The map $\bw\mapsto \set(\bw)$ is an embedding of $\cup_{k\geq 1}A^k$ to $\cK([0,1]^2)$. Let $d$ be the euclidean metric on $[0,1]^2$ and $\dha$ be the associated \emph{Hausdorff distance} on $\cK([0,1]^2)$. The space $(\cK[0,1]^2,\dha)$ is a compact metric space. Let $(\bbw_n)_{n\geq 1}$ be a sequence of words over $\{0,1\}$ with $|\bbw_n|\rightarrow\infty$. Basic topological considerations yield that $\rho_{\bw_n}$ converges weakly to $\rho\in\cC(\{0,1\})$ if and only if $\set(\bbw_n)$ converges with respect to $\dha$ to
	\begin{equation}\label{eq:setrho}
	\set(\rho):=\Big\{\big(\rho(\{0\}\times[0,t]),\rho(\{1\}\times[0,t])\big):0\leq t\leq 1\Big\}.
	\end{equation}
	
	\begin{exmp}
		Let $U,V$ be independent $\sim\unif[0,1]$ and let $\rho:=\law(1(V\leq U),U)\in\cC(\{0,1\})$. For each $0\leq t\leq 1$ it holds that 
		$$\rho(\{1\}\times[0,t])=\bP(V\leq U,U\leq t)=\int_0^t\bP(V\leq s)ds=\frac{t^2}{2}~~~\text{and}~~~\rho(\{0\}\times[0,t])=t-\frac{t^2}{2}.$$
		We simulated three words $\bw^1, \bw^2, \bw^3$ of length $80$ with distribution $\spread(\rho,80)$ independently and show $\set(\rho), \set(\bw^i), 1\leq i\leq 3$ in Figure~\ref{figg}.
		\begin{figure}
			\includegraphics[scale=0.55]{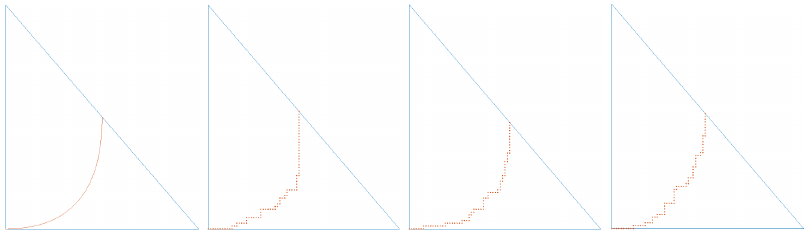}
			\caption{On the left $\set(\rho)$, then $\set(\bw^i)$ with $\bw^i\sim\spread(\rho,80)$.}\label{figg}
		\end{figure}
	\end{exmp}

	\subsection{Filtration Theory (Theorem 3)} In this part we investigate the backward filtrations $\bbF=(\cF_n)_{n\geq 1}$ generated by GEWPs. These filtrations fall into the class of so-called \emph{poly-adic (backward) filtrations}. As we explain below, interesting phenomena can occur when investigating the behavior of such filtrations 'near time $n=\infty$'. Initiated by A.~Vershik in the late 1960s, a rich mathematical theory has been developed around these type of (backward) filtrations. We refer the reader to Leuridan \cite{leuridan} for a more thorough introduction to this topic. In that paper one can also find an application of Laurents (and hence of our) results concerning the filtrations generated by EWPs. One does not need any prior knowledge about filtration theory to understand the following definitions.
	
	The literature about this topic usually deals with \emph{filtrations indexed by the negative integers}, i.e. increasing sequences of $\sigma$-fields $\cdots\subseteq\cF_{-3}\subseteq \cF_{-2}\subseteq\cF_{-1}$, whereas we consider \emph{backward filtrations indexed by the positive integers}, i.e. decreasing sequences of $\sigma$-fields $\cF_1\supseteq \cF_2\supseteq \cF_3\supseteq\cdots$. By defining $\cF'_n:=\cF_{-n}$ one can translate one situation to the other and hence both are equivalent settings. Since we have indexed GEWPs by positive integers, we choose to talk about backward filtrations indexed by $\bN$. The properties of backward filtrations we are interested in are in fact properties of filtered probability spaces, i.e. are in general not stable under change of measure. In particular, if the backward filtration under consideration is generated by a stochastic process, the properties we are interested in depend on the law of the process, not on the concrete probability spaces the process lives on. We now present the main definitions concerning poly-adic filtrations:
	
	\begin{definition}
		Let $(\Omega,\cA,\bP)$ be a probability space and let $\bbF=(\cF_n)_{n\geq 1}$ be a backward filtration, i.e. a decreasing sequence of $\sigma$-fields $\cA\supseteq\cF_1\supseteq\cF_2\supseteq\cdots$. 
		\begin{enumerate}
			\item[(i)] $\bbF$ is called \emph{kolmogorovian} if $\cF_{\infty}:=\cap_{n\geq 1} \cF_n$ is almost surely trivial.
			\item[(ii)] $\bbF$ is called \emph{poly-adic} if there exists stochastic process $\bbeta$, defined on $(\Omega,\cA,\bP)$, such that for each $n$
				\begin{enumerate}
					\item[(a)] $\eta_n$ is independent from $\cF_n$,
					\item[(b)] $\cF_n\as\sigma(\eta_{n+1})\vee\cF_{n+1}$,
					\item[(c)] $\eta_n$ is uniformly distributed on a finite set.
				\end{enumerate}
				$\bbeta$ is called a \emph{process of local innovations} for $\bbF$. If $\eta_n$ is uniform on a set with $r_n\in\bN$ elements, the term 'poly' is specified and $\bbF$ is called $(r_n)$-adic.
			\item[(iii)] $\bbF$ is called of \emph{product-type} if there exists a sequence of \emph{independent} RVs $\bbeta$ that \emph{generate} $\bbF$, i.e. such that $\cF_n\as\sigma(\eta_k:k\geq n+1)$ holds for all $n$.
		\end{enumerate}
	\end{definition}
	\begin{remark}
		The term '$(r_n)$-adic' is well defined: If $\bbF$ is a poly-adic filtration, $\bbeta, \bbeta^*$ are two processes of local innovations for $\bbF$ and $\eta_n$ is uniformly distributed on a set with $r_n$ elements, so is $\eta^*_n$.
	\end{remark}
	\begin{remark}
		$\eta_1$ plays an insignificant role in our definitions and in fact, the process of local innovations is usually shifted in the literature, i.e. $(\eta_{k+1})_{k\geq 1}$ is considered instead of $(\eta_k)_{k\geq 1}$. We decided to not shift $\bbeta$, since it is more convenient to have $\eta_n\in[n]$ instead of $\eta_n\in[n+1]$ for most part of our studies. 
	\end{remark}
	If $(\bbW,\bbeta)$ is a GEWP, then $\bbF=(\cF_n)_{n\geq 1}$ with $\cF_n:=\sigma(\bbW_k,\eta_{k+1}:k\geq n)$ is poly-adic due to the eraser-process $\bbeta$: (ii, a, c) follow immediately from Definition~\ref{def:gewp} and since $\cF_n=\sigma(\bbW_n)\vee\sigma(\eta_{n+1})\vee\cF_{n+1}$ and $\bbW_n=\erase(\bbW_{n+1},\eta_{n+1})$ almost surely, it follows that $\sigma(\bbW_n)\subas\sigma(\eta_{n+1})\vee\cF_{n+1}$ and hence (ii, b) holds. Since $\eta_n$ is uniform on the finite set $[n]$, the filtration $\bbF$ is $(n)$-adic.\\
	
	To get an idea of what the study of poly-adic filtrations is about, consider a poly-adic filtration $\bbF$ with process of local innovations $\bbeta$. Inductively applying property (ii, b) starting from $n=1$ yields
	\begin{equation*}
		\cF_1\as\sigma(\eta_2,\eta_3,\dots,\eta_n)\vee\cF_n~~\text{for each $n$}.
	\end{equation*}
	Again by property (ii, b) we see that $\sigma(\eta_m)\subas\cF_n$ for all $m>n$ and hence
	\begin{equation*}
		\cF_1\as\sigma(\eta_k:k\geq 2)\vee\cF_n~~\text{for each $n$}.
	\end{equation*}
	Taking the intersection over all $n$ on the right hand side yields 
	\begin{equation*}
		\cF_1\as\largecap_{n\geq 1}[\sigma(\eta_k:k\geq 2)\vee\cF_n].
	\end{equation*}
	Since the intersection is taken over all $n$ and the term $\sigma(\eta_k:k\geq 2)$ does not depend on $n$, one may wonders if it is allowed to \emph{interchange the order of taking intersection $\cap$ and supremum $\vee$} in this case. That is, one asks if 
	\begin{equation}\label{eq:int}
		\cF_1\as\largecap_{n\geq 1}[\sigma(\eta_k:k\geq 2)\vee\cF_n]\overset{\color{red}?\color{black}}{\as}\sigma(\eta_k:k\geq 2)\vee\largecap_{n\geq 1}\cF_n=\sigma(\eta_k:k\geq 2)\vee\cF_{\infty}
	\end{equation}
	holds. \emph{This is in general not the case}, we will see examples of this later. 
	\begin{remark}
		The question when it is allowed to interchange $\cap$ and $\vee$ was studied by von~Weizsäcker \cite{weiszaecker} in a very general setting, not just in the context of poly-adic filtrations. He presented some equivalent conditions for when the interchange is allowed. However, as these conditions are very abstract and stated in a very general setting, we do not see a way to apply them in our studies.	
	\end{remark}
	Suppose it would be allowed to interchange $\cap$ and $\vee$ in (\ref{eq:int}), i.e. suppose $\cF_1\as\sigma(\eta_k:k\geq 2)\vee \cF_{\infty}$ holds. If $\bbF$ would additionally be kolmogorovian, we would obtain $\cF_1\as\sigma(\eta_k:k\geq 2)$ and so $\bbF$ would be generated by $\bbeta$, hence be of product-type. The theory of poly-adic filtrations goes far beyond the question if the interchange of $\cap$ and $\vee$ is allowed for a concrete process of local innovations: One is interested if a generating processes of local innovations exist at all. We state some facts and refer the reader to \cite{leuridan, laurent} for details
	\begin{enumerate}
		\item Processes of local innovations are not unique.
		\item It may be that a process of local innovations $\bbeta$ does not generate $\bbF$ although there exists some other process of local innovations $\bbeta^*$ that does. Equivalent: $\bbF$ being of product-type does not imply that \emph{every} process of local innovations is generating.
		\item By Kolmogorov's zero-one law a necessary condition for a poly-adic filtration to be of product type is that it is kolmogorovian (hence the name).
		\item \emph{There exist kolmogorovian poly-adic filtrations that are not of product-type.} 
		\item There are known equivalent conditions for a poly-adic filtration being of product-type, but it would go beyond the scope of this paper to present them here. We refer the reader to \cite{laurent2}.
	\end{enumerate}
	Given a GEWP $(\bbW,\bbeta)$ we are interested in the question if the $(n)$-adic backward filtration $\bbF$ generated by $(\bbW,\bbeta)$ is of product-type or not. By (3) and the fact that $\bbF$ is kolmogorovian iff $(\bbW,\bbeta)$ is ergodic, we only consider ergodic GEWPs. 
	
	A first natural question coming up is to ask if the eraser-process $\bbeta$ may already be generating $\bbF$. This is answered by the following proposition, which we proof in Section~5:
	\begin{proposition}\label{prop:filtrationgeneratedbyusual}
		The backward filtration generated by an ergodic GEWP is generated by the eraser-process if and only if the law of the GEWP is given by $\GL(\rho)$ where $\rho\in\cC(A)$ is of the form $\rho=\law(f(U),U)$ for some measurable function $f:[0,1]\rightarrow A$ and $U\sim\unif[0,1]$.
	\end{proposition}
	In particular, we find that ergodic GEWPs whose laws are of the form $\GL(\cL(f(U),U))$ generate poly-adic filtrations of product-type. Since not all ergodic GEWPs are of this form, we need to continue our studies. 
	
	We proceed by considering EWPs, i.e. GEWPs in which the letters in $\bbW_n$ are iid for each $n$. It is easy to check that a GEWP over a Borel space $A$ is an EWP if and only if it is ergodic and the representing measure $\rho$ is of the form $\rho=\mu\otimes\unif[0,1]$ for some probability measure $\mu\in\cM_1(A)$. Note that $\law(W_{n,i})=\mu$ for all $1\leq i\leq n$ in this case. As long as $\mu$ is not a Dirac measure, EWPs can not be handled with Proposition~\ref{prop:filtrationgeneratedbyusual}. As we have already mentioned, the case of EWPs has been studied by Laurent \cite{laurent}, he obtained the following 
	\begin{theoremWITHOUT}[\cite{laurent}, Theorem 1]
		Every EWP $(\bbW,\bbeta)$ such that $(A,\mu)$ with $\mu=\cL(W_{1,1})$ is a Lebesgue probability space generates a backward filtration of product-type.
	\end{theoremWITHOUT}
	Laurent has proceeded his proof in three steps: First he considered finite alphabets $A$ and the case in which the letters of $\bbW_n$ are independent uniform over $A$. To prove that such EWPs generate product-type backward filtrations, he used a 'bare hands approach', i.e. he did not check an abstract criterion showing this property, but he constructed generating processes of local innovations using a limiting approach more or less explicitly. These constructions rely on the fact that $\bbW_n$ is uniform over $A^n$. The second step was to consider the case $A=[0,1]$ and the letters $W_{n,1},\dots,W_{n,n}$ being iid $\sim\mu$, where $\mu$ is the Lebesgue measure on $A$. By partitioning $[0,1]$ into equal length intervals he was able to reduce this to the first case. The last step was to consider arbitrary Lebesgue probability spaces $(A,\mu)$ and using that one can obtain $\mu$ as a push-forward of the Lebesgue measure under a measurable function $f:[0,1]\rightarrow A$. 
	
	Not all ergodic GEWPs are covered by Proposition~\ref{prop:filtrationgeneratedbyusual} and Laurent's result. We close the gap by showing 
	\begin{theorem}\label{thm:main3}
		Every ergodic GEWP over a Borel space $A$ generates a backward filtration of product-type.
	\end{theorem}
	The proof we present also proceeds in three steps: First we consider finite alphabets, then we consider $A=[0,1]$ and finally arbitrary Borel spaces. Like with Laurent, the main effort lies in the first step. For any ergodic GEWP $(\bbW,\bbeta)$ over a finite alphabet $A$ we will explicitly construct a generating processes $\bbeta^*$. Our construction involves the representing measure $\rho\in\cC(A)$, hence relies on Theorem~\ref{thm:main1}. Unlike the construction in \cite{laurent} we are able to express each $\bbW_n$ as a function of $(\eta^*_{n+1},\eta^*_{n+2},\dots)$ almost surely, i.e. our argument that the process $\bbeta^*$ is generating $\bbF$ is highly explicit. Our construction also sheds some light on the exchangeability point of view as we will see that our construction leads to triples $(\bbW,\bbeta,\bbeta^*)$ that correspond to certain jointly exchangeable triples $(\bbY,L,L^*)$, where $\bbY$ is a $A$-valued stochastic process and $L, L^*$ are random linear orders. The second and third step of the proof are very similar to~\cite{laurent}.


	\subsection{Outlook} One can introduce all kinds of erased-type processes and investigate the same topics as we have presented here in the case of words. As an example, we introduce erased-graph processes and finish with an open question concerning the backward filtrations these processes generate. Let $G$ be a (simple) graph on the node set $[n]$ and let $i\in[n]$. Consider the following procedure to obtain a graph on node set $[n-1]$: First, remove node $i$ from $G$ together with all adjacent edges. The result of this is a graph with node set $[n]\setminus\{i\}$. Now decrease the label of each node $j>i$ by one, i.e. $i+1\mapsto i, i+2\mapsto i+1, \dots, n\mapsto n-1$. The resulting graph has node set $[n-1]$, we denote this graph by $\erase(G,i)$. We omit 'general' in the following definition: 
	\begin{definition}\label{def:egp}
		An \emph{erased-graph process} (EGP) is a stochastic process $(\bbG,\bbeta)=(G_n,\eta_n)_{n\geq 1}$ such that for each $n$
		\begin{enumerate}
			\item[(i)] $G_n$ is a random graph on node set $[n]$,
			\item[(ii)] $\eta_n$ is uniform on $[n]$ and independent of the $\sigma$-field $\cF_n:=\sigma(G_k,\eta_{k+1}:k\geq n)$,
			\item[(iii)] $G_n=\erase(G_{n+1},\eta_{n+1})$ almost surely.
		\end{enumerate}
	\end{definition}
	One can investigate the representation, Martin boundary and filtration theory of EGPs like we did for GEWPs, the connections between these theories remain valid. Analogously to GEWPs one obtains that EGPs are in one-to-one correspondence with jointly exchangeable pairs $(G_{\infty},L)$ where $G_{\infty}$ is a random graph on node set $\bN$ and $L$ is a random linear order on $\bN$. A bijective description of ergodic exchangeable pairs $(G_{\infty},L)$ is hard to obtain, as this is already the case for exchangeable random graphs $G_{\infty}$. See \cite{austin, diaconisjanson} for material about the connection of exchangeable random graphs and graph limits. We finish with the following question:
	\begin{question}
		Are backward filtrations generated by ergodic EGPs of product-type? 
	\end{question}
	This is certainly true for ergodic EGPs that can be constructed from 'random-free graphons' (compare to Proposition~\ref{prop:filtrationgeneratedbyusual}), but we do not know if it is true for every ergodic EGP.
	
	\begin{remark}
		\cite{egw, ge} studied erased-type processes int the context of binary trees, Schröder trees and so-called interval-systems. \cite{ge} shows that ergodic erased-type processes generated poly-adic filtrations of product-type in this situation. 
	\end{remark}

	\subsection{Acknowledgments} The author would like to thank his PhD supervisor
	Rudolf Grübel for countless interesting discussions during the last years and for many very
	helpful comments concerning this paper and also Ludwig Baringhaus for pointing us to the concept of induced order statistics. The author would also like to thank two anonymous referees for their advises and encouragements that led to a substantial improvement both in terms of presentation and content of the material.

	\section{Connection to Exchangeability}
	
	We give a short introduction to the basic concepts of exchangeability theory, we refer the reader to \cite{kall}, Theorem A1.4 and \cite{dynkin} for details.
	
	Exchangeability theory is about the study of probability measures that are invariant with respect to the action of a permutation group. Let $\bS_{\infty}$ be the group of all permutations of $\bN$ that are finite in the sense that for each $\pi\in\bS_{\infty}$ one has $\pi(i)=i$ for all but finitely many $i$. Let $S$ be a Borel space and $\bS_{\infty}\times S\rightarrow S, (\pi,x)\mapsto x^{\pi}$ be some measurable group action from $\bS_{\infty}$ on $S$, i.e. it holds that $x^{\pi\circ\sigma}=(x^{\sigma})^{\pi}, x^{\id}=x$ and the map $x\mapsto x^{\pi}$ is measurable. A $S$-valued random variables $X$ is called \emph{exchangeable} if $\law(X)=\law(X^{\pi})$ for all $\pi\in\bS_{\infty}$. Let 
	\begin{equation*}
		\cME(S):=\{\law(X):~\text{$X$ is $S$-valued exchangeable}\}\subseteq \cM_1(S)
	\end{equation*}
	be the space of exchangeable laws. $\cME(S)$ is a simplex due to the following famous decomposition theorem: Let $\cI$ be the $\sigma$-field of $\bS_{\infty}$-invariant events, i.e. all measurable $B\subseteq S$ such that $B^{\pi}=\{x^{\pi}:x\in B\}=B$ for each $\pi\in\bS_{\infty}$. An exchangeable $X$ is called \emph{ergodic} if $\bP(X\in B)\in\{0,1\}$ for all $B\in\cI$. We denote by $\erg\cME(S)\subseteq\cME(S)$ the set of ergodic exchangeable laws. The extreme points of $\cME(S)$ are precisely given by $\erg\cME(S)$ and for each exchangeable $X$ the conditional law $\law(X|X^{-1}(\cI))$ is almost surely ergodic. With $\alpha:=\law(\law(X|X^{-1}(\cI)))$ one obtains the unique representation  $\law(X)=\int_{\erg\cME(S)}Qd\alpha(Q)$ (in the language of \cite{dynkin}, $\cI$ is $H$-sufficient for $\cME(S)$). We show below that the space $\cM(A)$ of laws of GEWPs over some Borel alphabet $A$ is affinely isomorphic to a space of the form $\cME(S)$ and by that decomposition theorems of $\cME(S)$ transfer to $\cM(A)$. 
	
	In a concrete case one is interested in finding a description, i.e. a parametrization, of $\erg\cME(S)$. We consider three cases: sequences (exchangeable processes), linear orders and pairs of sequences and linear orders. 
	
	\subsection{Exchangeable Processes}
	
	We consider a Borel space $A, S=A^{\bN}$ and the group action $\bS_{\infty}\times A^{\bN}\rightarrow A^{\bN}$ given by 
	\begin{equation*}
	(\pi,\bx)\in \bS_{\infty}\times A^{\bN}~\mapsto~\bx^{\pi}:=(x_{\pi^{-1}(k)})_{k\geq 1}.
	\end{equation*}
	A $A$-valued stochastic process $\bbX=(X_k)_{k\geq 1}$ is called exchangeable if $\cL(\bbX^{\pi})=\cL(\bbX)$ for $\pi\in\bS_{\infty}$. \emph{De Finetti's Theorem} describes the simplex structure of $\cME(A^{\bN})$: \emph{a Borel space-valued stochastic process is exchangeable iff it is mixed iid}. In particular, the ergodic exchangeable processes are precisely the iid processes, hence if we write $\rho^{\otimes \bN}$ for the law of an iid sequence with marginal $\rho$, one has 
	$\erg\cME(A^{\bN})=\{\rho^{\otimes\bN}:\rho\in\cM_1(A)\}$. We state an equivalent formulation involving random directing measures:
	\begin{theoremWITHOUT}[De Finetti, \cite{kall}]
		$\bbX=(X_j)_{j\geq 1}$ is exchangeable if and only if there exists a random probability measure $\Xi$ on $A$, i.e. a $\cM_1(A)$-valued random variable, such that 
		\begin{equation}\label{eq:definettias}
		\law(\bbX~|\Xi)=\Xi^{\otimes \bN}~~\text{almost surely}.
		\end{equation}
		$\Xi$ is called the random directing measure of $\bbX$ and is almost surely unique. For each event $B\subseteq A$ is holds that
		\begin{equation*}
			\Xi(B)=\lim\limits_{n\rightarrow\infty}\frac{1}{n}\sum_{j=1}^n1(X_j\in B)~~\text{almost surely}.
		\end{equation*}
		If $A$ is a polish space one has
		\begin{equation*}
		\Xi=\lim\limits_{n\rightarrow\infty}\frac{1}{n}\sum_{j=1}^n\delta_{X_j}~~\text{almost surely weakly.}
		\end{equation*}
	\end{theoremWITHOUT}
	Taking expectations in (\ref{eq:definettias}) yields the formulation 'exchangeable iff mixed iid'. Later we work with product spaces of the form $A=A_1\times A_2$, where both $A_1$ and $A_2$ are Borel. We will need the following
	\begin{lemma}\label{lemma:productspacedefinetti}
		Let $\bbX=(\bbY,\bbZ)=(Y_j,Z_j)_{j\geq 1}$ be an $A_1\times A_2$-valued exchangeable process with random directing measure $\Xi$.
		\begin{enumerate}
			\item[(i)] $\law(\bbZ)=\mu^{\otimes\bN}$ for some $\mu\in\cM_1(A_2)$ if and only if $\Xi(A_1\times\cdot)=\mu(\cdot)$ almost surely. 
			\item[(ii)] For each $\mu\in\cM_1(A_2)$ the space $\{\law(\bbX):\bbX~\text{is exchangeable and}~\law(\bbZ)=\mu^{\otimes\bN}\}$ is a simplex and $\law(\bbX)$ is an extreme point iff $\bbX$ is iid and $\law(\bbZ)=\mu^{\otimes\bN}$.
		\end{enumerate}
	\end{lemma}

	\begin{proof}
		If $\bbX=(\bbY,\bbZ)$ is exchangeable with random directing measure $\Xi(\cdot)$, then $\bbZ$ is exchangeable with random directing measure $\Xi(A_1\times\cdot)$, from which (i) follows directly. Now let $\mu\in\cM_1(A_2)$ and consider the 
		set $\cK:=\{\law(\bbX):\bbX~\text{is exchangeable and $\cL(\bbZ)=\mu^{\otimes\bN}$}\}$. This set is clearly convex. (i) yields that $\law(\bbX)\in\cK$ iff $\law(\bbX)$ is mixed iid where the mixture is only over those marginals $\rho\in\cM_1(A_1\times A_2)$ in which the second marginal is $\mu$. 
	\end{proof}

	\subsection{Exchangeable Linear Orders}\label{subsection:linord}
	For every GEWP $(\bbW,\bbeta)$ the eraser-process $\bbeta=(\eta_n)_{n\geq 1}$ always has the same distribution: $\eta_1,\eta_2,\dots$ are independent and $\eta_n\sim\unif[n]$ for each $n$. In this subsection we introduce the exchangeable random linear order on $\bN$ and explain that eraser-processes and exchangeable linear orders are in some sense equivalent. The material we present here seems to be folklore, but we are not aware of any references presenting the material in a closed form. What we present here is important for all three parts of our studies. From an exchangeability point of view we now consider the case $S=\bL$, where $\bL$ is the set of linear orders on $\bN$, where linear order is defined as follows:
	
	Given a binary relation $l\subseteq E\times E$ on a set $E$ we write $xly$ instead of $(x,y)\in l$. A (strict) linear order on $E$ is a binary relation $l$ that is transitive (for all $x,y,z\in E$ it holds that $xly\wedge ylz~\Rightarrow~xlz$) and trichotomous (for all $x,y\in E$ excatly one of the three statements $xly$ or $ylx$ or $x=y$ is true). We write $\bL$ for the set of linear orders on $\bN$ and $\bL_n$ for the finite set of linear orders on $[n]$, $<\in\bL$ denotes the usual linear order on $\bN$, i.e. $1<2<3<\dots$. For $n\in\bN$ and $l\in\bL$ we denote by $l_{|n}\in\bL_n$ the restriction of $l$ to the set $[n]$. We endow $\bL$ with the $\sigma$-field generated by the projections $l\mapsto l_{|n}, n\in\bN$. With this $\bL$ becomes a Borel space. One can see this by noting that $d(l,l'):=\sum_n2^{-n}1(l_{|n}\neq l'_{|n})$ defines a metric on $\bL$ that turns $(\bL,d)$ into a compact metric space and that the associated Borel $\sigma$-field equals the $\sigma$-field generated by the restriction maps $l\mapsto l_{|n}$. The law $\law(L)$ of a $\bL$-valued random variable $L$ is determined by the sequence $(\law(L_{|n}))_{n\geq 1}$. Given some $l\in\bL$ and $\pi\in\bS_{\infty}$ we define a linear order $l^{\pi}$ by 
	$$i~l^{\pi}~j~~:\Longleftrightarrow~~\pi^{-1}(i)~l~\pi^{-1}(j)~~~\text{for all}~i,j\in\bN.$$
	The map $\bS_{\infty}\times\bL\rightarrow\bL, (\pi,l)\mapsto l^{\pi}$ yields a measurable group action. The representation result for exchangeable linear orders reads as follows:
	\begin{proposition}\label{thm:exchlin}
		A random linear order $L$ on $\bN$ is exchangeable, i.e. $\cL(L^{\pi})=\cL(L)$ for each $\pi\in\bS_{\infty}$, if and only if for each $n$ the restriction $L_{|n}$ is uniformly distributed on $\bL_n$. In particular, the law of an exchangeable linear order is unique and $\erg\cME(\bL)=\cME(L)$.
	\end{proposition}
	\begin{proof}
		The action of $\bS_n$ to $\bL_n$ defined by $il^{\pi}j:\Leftrightarrow\pi^{-1}(i)l\pi^{-1}(j)$ is transitive and hence any $\bL_n$-valued RV is exchangeable (w.r.t. $\bS_n$) iff it is uniformly distributed. Now let $\pi\in\bS_n$ and define $\tilde\pi\in\bS_{\infty}$ by $\tilde\pi(i):=i$ for $i>n$ and $\tilde\pi(i):=\pi(i)$ for $i\in[n]$. For each $l\in\bL$ it holds that $(l_{|n})^{\pi}=(l^{\tilde\pi})_{|n}$. Since $\law(L)$ is determined by $(\law(L_{|n}))_{n\geq 1}$ this yields the result.
	\end{proof}
	
	Representation results for exchangeable random objects are often stated in a form involving a stochastic processes $\bbU=(U_j)_{j\geq 1}$ in which $U_j$ are iid $\unif[0,1]$-distributed. We call such a processes a $U$-process. We explain that $U$-processes are basically equivalent to exchangeable linear orders: Given a $U$-process $\bbU=(U_j)_{j\geq 1}$ we define a random linear order by
	\begin{equation*}
		i~L~j~~:\Leftrightarrow~~U_i<U_j~~~\text{for all}~i,j\in\bN.
	\end{equation*}
	Since $\bbU$ is exchangeable, so is $L$. One can recover $\bbU$ from $L$ almost surely because
	\begin{equation*}
		U_i=\lim\limits_{n\rightarrow\infty}\frac{1}{n}\#\big\{k\in[n]:k~L~i\big\}
	\end{equation*}
	holds a.s. for each $i$ be the strong law of large numbers and hence $\sigma(\bbU)\as\sigma(L)$. We call $\bbU$ the $U$-process corresponding to $L$ and vice versa. There are two more stochastic objects that are equivalent to an exchangeable linear order (or $\bbU$-process) in this way: For each $n$ define 
	\begin{equation*}
		S_n:=\ps(U_1,\dots,U_n) ~~\text{(permutation statistics)},
	\end{equation*}
	i.e. $S_n$ is the unique random permutation of $[n]$ with $U_{S_n(1)}<\dots<U_{S_n(n)}$. Equivalently, if $L$ is the linear order corresponding to $\bbU$, then $S_n$ is the unique permutation of $[n]$ such that
	\begin{equation*}
		S_n(1)~L~S_n(2)~L~\dots~L~S_n(n).
	\end{equation*}
	For each $n$ $S_n$ contains the same information as $L_{|n}$. The distribution of the process $\bbS=(S_n)_{n\geq 1}$ is determined by the two properties
	\begin{enumerate}
		\item[(i)] $S_n$ is uniform on $\bS_n$ for each $n$,
		\item[(ii)] The one-line-notation of $S_n$ is obtained by erasing '$n+1$' from the one-line-notation of $S_{n+1}$ for each $n$, in symbols:
		\begin{equation*}
			\big(S_n(1),\dots,S_n(n)\big)=\erase\Big(\big(S_{n+1}(1),\dots,S_{n+1}(n+1)\big),S_{n+1}^{-1}(n+1)\Big).
		\end{equation*}
	\end{enumerate}
	We call a stochastic process $\bbS$ that fulfills (i) and (ii) a $S$-process. Note that $S_k=\ps(S_n^{-1}(1),\dots,S_n^{-1}(k))$ for each $k\leq n$.

	If $\bbS$ has been constructed from an $U$-process as above, one can recover $\bbU$ (and hence $L$) almost surely, since $S_n^{-1}(i)=1+\#\{k\in[n]:kLi\}$ and hence
	\begin{equation*}
		U_i = \lim\limits_{n\rightarrow\infty}\frac{1}{n}S_n^{-1}(i)
	\end{equation*}
	almost surely for each $i$. In particular, $\sigma(\bbU)\as\sigma(L)\as\sigma(\bbS)$. We call $\bbS$ the $S$-process corresponding to $L/\bbU$ and vice versa.

	As already noted at the beginning, the fourth object 'equivalent' to an exchangeable linear order is an eraser-process. If $(\bbU,L,\bbS)$ is a corresponding triple like before, define
	\begin{equation*}
		\eta_n:=\#\{k\in[n]:U_k\leq U_n\}=1+\#\{k\in[n]:kLn\}=S_n^{-1}(n). 
	\end{equation*}

	One can recover $\bbS=(S_n)_{n\geq 1}$ (and hece $\bbU$ and $L$) from $\bbeta$ by the following inductive procedure
	\begin{enumerate}
		\item $S_1:=(1)\in\bS_1$.
		\item The one-line-notation of $S_{n+1}$ is obtained by inserting '$n+1$' in the $\eta_{n+1}$-th slot in the one-line-notation of $S_n=(\square S_n(1)\square S_n(2)\square\dots \square S_n(n)\square)$.
	\end{enumerate}
	In particular, $(\eta_1,\dots,\eta_n)$ contains the same information as $S_n$. We call $\bbeta$ the eraser-process corresponding to $L/\bbU/\bbS$ and vice versa. So given any of the four objects under consideration, i.e. exchangeable linear order, $U$-process, $S$-process or eraser process, there are almost surely uniquely defined corresponding objects of the other three types given by the constructions presented above and all objects contain the same probabilistic information, i.e. $\sigma(\bbU)\as\sigma(L)\as\sigma(\bbS)\as\sigma(\bbeta)$ holds almost surely. 
	
	The following proposition will be used when studying the filtrations generated be GEWPs in Section~4 and is worth knowing when dealing with $(n)$-adic filtrations in general:
	\begin{proposition}\label{prop:eta}
		Let $\bbeta=(\eta_n)_{n\geq 1}$ be an eraser-process and let $\bbU=(U_j)_{j\geq 1}$ be the $U$-processes corresponding to $\bbeta$. Then for each $n$
		\begin{equation*}
			\sigma(\eta_k:k\geq n+1)\as\sigma(U_{1:n},\dots,U_{n:n})\vee\sigma(U_k:k\geq n+1).
		\end{equation*}
	\end{proposition}
	\begin{proof}
		Since $\sigma(S_n)=\sigma(\eta_1,\dots,\eta_n)$ we can write
		$$\sigma(\bbeta)=\sigma(S_n)\vee\sigma(\eta_k:k\geq n+1).$$
		It also holds that
		$$\sigma(\bbU)=\sigma(S_n)\vee\sigma(U_{1:n},\dots,U_{n:n})\vee\sigma(U_k:k\geq n+1).$$
		Since $\sigma(S_n)$ is clearly independent from $\sigma(\eta_k:k\geq n+1)\vee\sigma(U_{1:n},\dots,U_{n:n})\vee\sigma(U_k:k\geq n+1)$, one can remove $\sigma(S_n)$ on both sides of
		$$\sigma(S_n)\vee\sigma(\eta_k:k\geq n+1)\as\sigma(S_n)\vee\sigma(U_{1:n},\dots,U_{n:n})\vee\sigma(U_k:k\geq n+1)$$
		and the result follows.
	\end{proof}
	
	\subsection{Exchangeable Pairs of Processes and Linear Orders} 
	
	We consider the 'product' of Sections 3.1 and 3.2, i.e. $S=A^{\bN}\times\bL$, where $A$ is some Borel space, and the \emph{diagonal action}
	\begin{equation*}
	 	\bS_{\infty}\times A^{\bN}\times\bL~\rightarrow~A^{\bN}\times\bL,~~(\pi,\by,l)\mapsto(\by^{\pi},l^{\pi}).
	\end{equation*}
	An $A^{\bN}\times\bL$-valued RV $(\bbY,L)$ is called (jointly) exchangeable if $\law(\bbY^{\pi},L^{\pi})=\law(\bbY,L)$ for each $\pi\in\bS_{\infty}$. If $(\bbY,L)$ is exchangeable then both $\bbY$ and $L$ are exchangeable and so the results of subsections~3.1 and~3.2 apply to them individually. We are interested in the joint behavior of exchangeable $(\bbY,L)$.
	\begin{lemma}\label{lemma:help}
		Let $(\bbY,\bbU)=(Y_j,U_j)_{j\geq 1}$ be an $A\times\bR$-valued exchangeable process with $\law(\bbU)=\unif[0,1]^{\otimes\bN}$ and let $L$ be the linear order corresponding to $\bbU$. Then $(\bbY,L)$ is exchangeable and the map 
		\begin{align*}
			\Big\{\law(\bbY,\bbU):(\bbY,\bbU)~&\text{is exchangeable $A\times\bR$-valued and}~\law(\bbU)=\unif[0,1]^{\otimes\bN}\Big\}\\
											  &\longrightarrow \cME(A^{\bN}\times\bL),~~\law(\bbY,\bbU)\mapsto \law(\bbY,L)
		\end{align*}
		is an affine bijection.
	\end{lemma}
	\begin{proof}
		The map is clearly affine. For each $\pi\in\bS_{\infty}$ it holds that if $\bbU$ is the $U$-process corresponding to $L$ then $\bbU^{\pi}$ is the $U$-process corresponding to $L^{\pi}$. So $(\bbY,L)$ is exchangeable iff $(\bbY,\bbU)$ is exchangeable. Since one can recover $(\bbY,\bbU)$ from $(\bbY,L)$ and vice versa, the map $\law(\bbY,L)\mapsto \law(\bbY,\bbU)$ is one-to-one as claimed. 
	\end{proof}

	As a consequence of Lemmas~\ref{lemma:productspacedefinetti}~and~\ref{lemma:help}, we obtain a representation result for ergodic exchangeable pairs $(\bbY,L)$. Recall that we have defined $\cC(A)\subset\cM_1(A\times\bR)$ to be the set of all probability measures $\rho$ on $A\times\bR$ with $\rho(A\times\cdot)=\unif[0,1]$. 
	
	\begin{proposition}\label{prop:exchpair}
		Let $\rho\in\cC(A)$ and $(\bbY,\bbU)=(Y_j,U_j)_{j\geq 1}$ be iid $\sim\rho$. Let $L$ be the linear order corresponding to $\bbU$. Then $(\bbY,L)$ is ergodic exchangeable and the map $\rho\mapsto \law(\bbY,L)$ is a bijection between $\cC(A)$ and $\erg\cME(A^{\bN}\times\bL)$.
	\end{proposition}

	\begin{remark}\label{rem:pairs}
		One can use very similar arguments when considering the case $\cME(\bL^k)$ for some $k\geq 1$ and obtain that laws of ergodic exchangeable tuples $(L_1,\dots,L_k)$ are in one-to-one correspondence with $k$-dimensional copulas due to the fact that if $(U^i_j)_{j\geq 1}$ is the $U$-process corresponding to $L_i$, then $(U^1_j,\dots,U^k_j)_{j\geq 1}$ is iid. 
	\end{remark}

	\subsection{Proof of Theorem 1} Let $(\bbY,L)$ be $A^{\bN}\times\bL$-valued exchangeable and let $\bbS, \bbU, \bbeta$ be the $S$/$U$/eraser-processes corresponding to $L$. For each $n$ define 
	\begin{equation*}
		\bbW_n:=\ios(Y_1,\dots,Y_n,U_1,\dots,U_n)=(Y_{S_n(1)},\dots,Y_{S_n(n)})
	\end{equation*} 
	and $\bbW=(\bbW_n)_{n\geq 1}$. 
	\begin{lemma}\label{lemma:1}
		$(\bbW,\bbeta)$ is a GEWP.
	\end{lemma}
	\begin{proof}
		We need to check that $(\bbW,\bbeta)$ fulfills the assumptions (i)-(iii) from Definition~\ref{def:gewp}. Properties (i) and (iii) are obvious from the construction of $(\bbW,\bbeta)$. The first part of (ii), $\bbeta_n\sim\unif[n]$, follows since $\bbeta$ is the eraser-process corresponding to the exchangeable linear order $L$. The only thing left to check is that $\eta_k$ is independent from $(\bbW_n,\bbeta_{n+1})_{n\geq k}$ for each $k$. To see this, we first show that $\bbW_n$ and $S_n$ are independent for each $n$. Let $B\subseteq A^n$ be an event and let $\pi\in\bS_n$. Because $\bbW_n=(Y_{S_n(1)},\dots,Y_{S_n(n)})$ we have
		\begin{equation}\label{eq:key1}
			\bP(\bbW_n\in B, S_n=\pi)=\bP((Y_{\pi(1)},\dots,Y_{\pi(n)})\in B, S_n=\pi).
		\end{equation}
		As we have already seen in Lemma~\ref{lemma:help}, since $(\bbY,L)$ is exchangeable, so is $(Y_k,U_k)_{k\geq 1}$, where $\bbU=(U_k)_{k\geq 1}$ is the $U$-process corresponding to $L$. Now $S_n$ is the unique random permutation arranging the first $n$ $U$-values. Because of exchangeability of $(Y_k,U_k)_{k\geq 1}$ the expression in 
		(\ref{eq:key1}) does not depend on $\pi$. Summing over all $\pi\in\bS_n$ yields
		\begin{equation*}
			\bP(\bbW_n,S_n=\pi)=\bP(\bbW_n\in B)\frac{1}{n!}=\bP(\bbW_n\in B)\bP(S_n=\pi),
		\end{equation*}
		hence the independence of $\bbW_n$ and $S_n$. Now we want to show that $\eta_k$ is independent from $(\bbW_n,\bbeta_{n+1})_{n\geq k}$ for each $k$. It is enough to show that 
		\begin{equation*}
			\eta_k~~~\text{and}~~~(\bbW_k,\eta_{k+1},\bbW_{k+1},\dots,\eta_{n},\bbW_{n})
		\end{equation*}
		are independent for each $k\leq n$. Since $\bbW_k=\erase(\bbW_{k+1},\eta_{k+1})$ almost surely for each $k$, we only need to show that 
		\begin{equation*}
			\eta_k~~~\text{and}~~~(\eta_{k+1},\eta_{k+2},\dots,\eta_{n},\bbW_{n})
		\end{equation*}
		are independent. Since the eraser-process is a process of independent RVs we are done if we can show that $(\eta_1,\dots,\eta_{n})$ and $\bbW_{n}$
		are independent. But this follows from the fact that $(\eta_1,\dots,\eta_n)$ contains the same information as $S_n$ together with the independence form $\bbW_n$ and $S_n$.
	\end{proof}
	Now let $(\bbW,\bbeta)$ be a GEWP and let $L$ be exchangeable linear order corresponding to $\bbeta$. For each $j\geq 1$ we define 
	\begin{equation*}
		Y_j:=W_{j,\eta_j},
	\end{equation*}
	where $\bbW_n=(W_{n,1},\dots,W_{n,n})$. The random letter $Y_j\in A$ is the one that gets erased by passing from $\bbW_j$ to $\bbW_{j-1}$. Let $\bbY=(Y_j)_{j\geq 1}$. 
	\begin{lemma}\label{lemma:2}
		$(\bbY,L)$ is exchangeable.
	\end{lemma}	
	\begin{proof}
		We need to show that $\cL(\bbY^{\pi},L^{\pi})=\cL(\bbY,L)$ for each $\pi\in\bS_{\infty}$. Let $\bbU$ be the $U$-process corresponding to $L$. By Lemma~\ref{lemma:help} we need to show that $(Y_k,U_k)_{k\geq 1}$ is a $A\times\bR$-valued exchangeable process, hence we need to show that for each $k$ and each $\pi\in\bS_k$ 
		\begin{equation*}
			(Y_{\pi(1)},U_{\pi(1)},\dots,Y_{\pi(k)},U_{\pi(k)})~~~\text{and}~~~(Y_1,U_1,\dots,Y_k,U_k)
		\end{equation*}
		have the same distribution. Let $\bbS=(S_n)_{n\geq 1}$ be the $S$-process corresponding to $\bbeta$. Since $U_j=\lim_nn^{-1}S_n^{-1}(j)$ almost surely for each $j$ it is enough to show that
		\begin{equation}\label{eq:key}
			(Y_{\pi(1)},S_n^{-1}(\pi(1)),\dots,Y_{\pi(k)},S_n^{-1}(\pi(k)))~~~\text{and}~~~(Y_1,S_n^{-1}(1),\dots,Y_k,S_n^{-1}(k))
		\end{equation}
		have the same distribution for all $k\leq n$. Because $\bbW_n=\erase(\bbW_{n+1},\eta_{n+1})$ almost surely for each $n$ it holds that
		\begin{equation*}
			Y_j=W_{n,S_n^{-1}(j)}~~\text{almost surely for all}~j\leq n.
		\end{equation*}
		Hence the vector on the right side of (\ref{eq:key}) can be written as $h(\bbW_n,S_n)$ for a suitable function $h$ and the vector on the left side can be written as $h(\bbW_n,\pi^{-1}\circ S_n)$, where we have extended $\pi\in\bS_k$ to $\pi\in\bS_n$ by identity. As we have seen in the proof of Lemma~\ref{lemma:1}, since $(\bbW,\bbeta)$ is a GEWP, $\bbW_n$ and $S_n$ are independent for each $n$, moreover $\cL(S_n)=\cL(\pi^{-1}\circ S_n)=\unif(\bS_n)$. Hence $(\bbW_n,S_n)$ and $(\bbW_n,\pi^{-1}\circ S_n)$ have the same distribution, which finishes the proof.
	\end{proof}

	\begin{proof}[Proof of Theorem \ref{thm:main1}]
		Let $(\bbY,L)$ be a jointly exchangeable pair and let $(\bbW,\bbeta)$ be the GEWP constructed as in Lemma~\ref{lemma:1}. Lemma~\ref{lemma:2} and the fact that the constructions we have presented (passing from a GEWP to exchangeable pair $(\bbY,L)$ and vice versa) are inverse of each other, imply that the map $\law(\bbY,L)\mapsto\law(\bbW,\bbeta)$ is a bijection between $\cME(A^{\bN}\times\bL)$ and $\cM(A)$. Since this bijection is clearly affine, it maps ergodic laws bijectively to ergodic laws. The representation result in Theorem~\ref{thm:main1} hence follows from the representation result of ergodic exchangeable pairs $(\bbY,L)$, Proposition~\ref{prop:exchpair}.
	\end{proof}
	
	We finish Section~3 by taking a closer look to random directing measures associated to GEWPs. Let $(\bbW,\bbeta)$ be a GEWP over a Borel space $A$ and let $\bbY=(Y_j)_{j\geq 1}$ be the sequence of erased letters, i.e. $Y_j=W_{j,\eta_j}$ and let $\bbU=(U_j)_{j\geq 1}$ be the $U$-process corresponding to $\bbeta$. We have seen that the process $(\bbY,\bbU)=(Y_j,U_j)_{j\geq 1}$ is an $A\times\bR$-valued exchangeable process with $\law(\bbU)=\unif[0,1]^{\otimes\bN}$. Lemma~\ref{lemma:productspacedefinetti} and de~Finetti's~theorem yield that there exists an a.s. unique random directing measure $\Xi$ with $\bP(\Xi\in\cC(A))=1$ such that $\law(\bbY,\bbU|\Xi)=\Xi^{\otimes\bN}$ almost surely. This yields 
	\begin{equation*}
		\law(\bbW,\bbeta|\Xi)=\GL(\Xi)~~\text{almost surely},
	\end{equation*}
	see Theorem~\ref{thm:main1} for the definition of $\GL(\rho), \rho\in\cC(A)$. If $A$ is a polish space, so is $A\times\bR$ and we can obtain $\Xi$ as an almost sure weak limit by
	\begin{equation}\label{eq:app}
		\Xi = \lim\limits_{n\rightarrow\infty}\frac{1}{n}\sum_{j=1}^n\delta_{(Y_j,U_j)}~~\text{almost sure weakly}.
	\end{equation}
	Note that the $n$-th empirical measure $\frac{1}{n}\sum_{j=1}^n\delta_{(Y_j,U_j)}$ is \emph{not measurable with respect to $\sigma(\bbW_1,\eta_1,\dots,\bbW_n,\eta_n)$}. Changing the order of summation in the $n$-th empirical measure yields 
	\begin{equation}\label{eq:empiricalmeasure}
		\frac{1}{n}\sum_{j=1}^n\delta_{(Y_j,U_j)}=\frac{1}{n}\sum_{j=1}^n\delta_{(W_{n,j},U_{j:n})}.
	\end{equation}
	Since $\bbU=(U_j)_{j\geq 1}$ is a $U$-process, $\sup_{0\leq t\leq 1}|U_{\lfloor nt\rfloor:n}-t|\rightarrow 0$ almost surely. It is thus reasonable to replace $U_{j:n}$ with $j/n$ in (\ref{eq:empiricalmeasure}). We get an alternative approximation of $\Xi$ that is measurable with respect to $\sigma(\bbW_n)$:
	\begin{equation*}
		\frac{1}{n}\sum_{j=1}^n\delta_{(W_{n,j},\frac{j}{n})}=\rho_{\bbW_n},
	\end{equation*}
	where $\rho_{\bw}, \bbw\in \cup_kA^k$ has been introduced in Section 2.2. As a consequence of Theorem~2 we later obtain 
	\begin{equation*}
		\Xi=\lim\limits_{n\rightarrow\infty}\frac{1}{n}\sum_{j=1}^n\delta_{(W_{n,j},\frac{j}{n})}~~\text{almost sure weakly}.
	\end{equation*}
	
	\section{Proof of Theorem 2 via Coupling Methods}
	
	In this section we consider polish alphabets $A$. Our goal is to prove Theorem~\ref{thm:main2}. First we introduce the ($1^{\text{th}}$-)Wasserstein distance. If $S$ is a polish space then there exists a metric $d$ on $S$ that is compatible with $S$ and such that $(S,d)$ is a complete separable metric space. One can choose $d$ to be bounded. If $d$ is 
	such a metric, we define the associated \emph{Wasserstein metric} (also called \emph{Kantorovich-Rubinstein metric}) $\dW$ on $\cM_1(S)$ by
	\begin{align*}
		\dW(\mu,\nu):&=\inf\Big\{\bE(d(X,Y)): (X,Y)~\text{is $S\times S$-valued RV with $X\sim\mu, Y\sim\nu$}\Big\}\\
		&=\inf\limits_{\psi~\text{coupling of}~(\mu,\nu)}\int_{S\times S}d(x,y)d\psi(x,y).
	\end{align*}
	The space $(\cM_1(S),\dW)$ is again a complete separable metric space, $\dW$ is bounded and the topology generated by $\dW$ coincides with the weak topology on $\cM_1(S)$. Recall that the Borel $\sigma$-field on $\cM_1(S)$ coincides with the $\sigma$-field we have considered before, i.e. is generated by $P\mapsto P(B)$, $B\subset S$ measurable. On each of the polish spaces $A^k, \bR$ we choose bounded compatible metric $d_k, d_{\bR}$ like explained. On the product $A\times\bR$ we consider the metric $d((y,u),(y',u')):=d_1(y,y')+d_{\bR}(u,u')$. The associated Wasserstein metrics on $\cM_1(A^k)$ and $\cM_1(A\times\bR)$ are denoted by $\dW_k$ and $\dW$.
	
	\begin{lemma}\label{lemma:11}
		Let $k\geq 1$. 
		\begin{enumerate}
			\item[(i)] $\sample:\cM_1(A^k)\rightarrow\cM_1(A\times\bR)$ is continuous.
			\item[(ii)] $\spread(\cdot,k):\cM_1(A\times\bR)\rightarrow\cM_1(A^k)$ is continuous at $\rho\in\cM_1(A\times\bR)$ with a continuous second marginal distribution.
		\end{enumerate}
	\end{lemma}
	\begin{proof}
		(i) Let $\mu_n,\mu\in\cM_1(A^k)$. $\mu_n\rightarrow\mu$ implies $\mu_n\otimes\unif[k]\rightarrow\mu\otimes\unif[k]$ weakly in $\cM_1(A^k\times[k])$. The map $(\bbw,j)\mapsto (w_j,j/k)$ is continuous and $\sample(\mu_n)$ is the push-forward of $\mu_n\otimes\unif[k]$ under this map. Hence continuity of $\sample$ follows from continuous mapping theorem.\\
		(ii) The map $\ios:A^k\times\bR^k\rightarrow A^k$ is continuous at points $(y_1,\dots,y_k,x_1,\dots,x_k)$ with $x_i\neq x_j$ for $i\neq j$. If $(Y_1,X_1),\dots,(Y_k,X_k)$ are iid $\sim\rho$ and the second marginal of $\rho$ is continuous, then $X_i\neq X_j$ for all $i\neq j$ almost surely. Since $\spread(\rho,k)=\law(\ios(Y_1,\dots,Y_k,X_1,\dots,X_k))$, the result again follows from continuous mapping theorem.
	\end{proof}

 	\begin{lemma}\label{lemma:22}
 		Let $\bw\in A^n$. 
 		\begin{enumerate}
 			\item[(i)] Let $k\leq n$ and $C_k$ be a constant with $d_k\leq C_k$. Then
 			\begin{equation*}
	 			\dW_k(\RSS(\bbw,k),\spread(\rho_{\bbw},k))\leq C_k\cdot \Big[1-\frac{n!}{(n-k)!n^k}\Big].
 			\end{equation*}
 			For each $k$ the upper bound converges to zero as $n=|\bw|\rightarrow\infty$.
 			\item[(ii)] Let $k, m\leq n$ and let $\bbS=(S_j)_{j\geq 1}$ be a $S$-process. Then
 			\begin{equation*}
 				\dW\Big(\sample(\RSS(\bbw,k)),\sample(\RSS(\bbw,m))\Big)\leq \bE\Big(d_{\bR}\Big(\frac{S_k^{-1}(1)}{k},\frac{S_m^{-1}(1)}{m}\Big)\Big).
 			\end{equation*}
 			As $\min\{k,m\}\rightarrow \infty$ the upper bound converges to zero.
 		\end{enumerate}
 		
 	\end{lemma}
 	\begin{proof}
 		(i) We construct a coupling of $\RSS(\bbw,k)$ and $\spread(\rho_{\bbw},k)$. Let $(Y_1,J_1),\dots,(Y_k,J_k)$ be iid $\sim\rho_{\bbw}$. Define $\bbX:=\ios(Y_1,\dots,Y_k,J_1,,\dots,J_k)$. By definition, $\bbX$ has law $\spread(\rho_{\bbw},k)$. Now let $\bbX^*\sim \RSS(\bbw,k)$ be independent of $\bbX$ and define 
 		\begin{equation*}
 			\bbX':=\begin{cases}
 				\bbX,&~\text{if}~J_i\neq J_j~\text{for all}~i\neq j\\
 				\bbX^*,&~\text{else}.
 			\end{cases}
 		\end{equation*}
 		It is easy to see that $\bbX'$ has law $\RSS(\bbw,k)$ and hence $(\bbX',\bbX)$ is a coupling of $\RSS(\bbw,k)$ and $\spread(\rho_{\bbw},k)$. We obtain
 		\begin{align*}
 		\dW_k(\RSS(\bbw,k),\spread(&\rho_{\bbw},k))\\
 		&\leq \bE(d_k(\bbX',\bbX))\\
 		&=\bE\Big(d_k(\bbX,\bbX)\cdot 1(\text{$J_i\neq J_j$ for all $i\neq j$})\Big)\\
 		&~~~~~~~~~~~~~~~~ + \bE\Big(d_k(\bbX^*,\bbX)\cdot 1(\text{$J_i= J_j$ for some $i\neq j$})\Big)\\
 		&\leq C_k\cdot \bP(\text{$J_i= J_j$ for some $i\neq j$})=C_k\cdot \Big[1-\frac{n!}{(n-k)!n^k}\Big].
 		\end{align*}
 		(ii) We again construct a coupling. Let $\bbS=(S_j)_{j\geq 1}$ be a $S$-process and let $\bbeta=(\eta_n)_{n\geq 1}$ be the eraser-process corresponding to $\bbS$. We define $\bbW_n:=\bbw\in A^n$ and by induction $\bbW_{j-1}:=\erase(\bbW_j,\eta_j)$ for $n\geq j\geq 2$. For all $1\leq j\leq N$ it holds that $\bbW_j=(W_{j,1},\dots,W_{j,j})$ has distribution $\RSS(\bbw,j)$. In particular, $(\bbW_k,\bbW_m)$ is a coupling of $\RSS(\bbw,k)$ and $\RSS(\bbw,m)$. Like in the proof of Lemma~\ref{lemma:1} it is easy to see that for each $1\leq j\leq n$ the permutation $S_j$ is independent from $\bbW_j$. Since $S_j$ is uniform on $\bS_j$, we get that $S_j^{-1}(1)\sim\unif[j]$ and hence $(W_{j,S_j^{-1}(1)},S_j^{-1}(1)/j)$ has distribution $\sample(\RSS(\bbw,j))$ for each $j$. We define 
 		\begin{equation*}
 		\bbX:=\Big(W_{k,S_k^{-1}(1)},\frac{S_k^{-1}(1)}{k}\Big)~~\text{and}~~\bbX':=\Big(W_{m,S_m^{-1}(1)},\frac{S_m^{-1}(1)}{m}\Big).
 		\end{equation*}
 		$(\bbX,\bbX')$ is a coupling of $\sample(\RSS(\bbw,k))$ and $\sample(\RSS(\bbw,m))$. Recall that we have defined $d((y,u),(y',u'))=d_1(y,y')+d_{\bR}(u,u')$. This yields
 		\begin{align*}
 		\dW\Big(\sample(\RSS(\bbw,k))&,\sample(\RSS(\bbw,m))\Big)\\
 		&\leq \bE(d(\bbX,\bbX'))\\
 		&=\bE(d_1(W_{k,S_k^{-1}(1)},W_{m,S_m^{-1}(1)}))~+~\bE\Big(d_{\bR}\Big(\frac{S_k^{-1}(1)}{k},\frac{S_m^{-1}(1)}{m}\Big)\Big).
 		\end{align*}
 		Algorithmic properties of $\erase$ yield $W_{j,S_j^{-1}(1)}=W_{1,1}$ for all $1\leq j\leq n$, and hence the first term in the upper bound is zero.\\
 		To see that the expectation converges to zero as $k\rightarrow\infty$, recall that $\frac{1}{k}S_k^{-1}(1)$ converges almost surely to $U_1$, where $\bbU$ is the $U$-process corresponding to $\bbS$. Because $d_{\bR}$ is bounded, the statement follows from dominated convergence.
 	\end{proof}


	\begin{lemma}\label{lemma:33}
		If $(\bbw_n)_{n\geq 1}$ is a $\RSS$-convergent sequence with limit $\bmu=(\mu_k)_{k\geq 1}$, then there exists some $\rho\in\cC(A)$ with
		$\sample(\mu_k)\rightarrow\rho$ as $k\rightarrow\infty$.
	\end{lemma}
	\begin{proof}
		Suppose $(\sample(\mu_k))_{k\geq 1}$ is a Cauchy sequence in $\cM_1(A\times\bR)$. By completeness of $\cM_1(A\times\bR)$ there exists a $\rho\in\cM_1(A\times\bR)$ with $\sample(\mu_k)\rightarrow\rho$. The second marginal of 
		$\sample(\mu_k)$ is given by $\law(J_k/k)$ with $J_k\sim\unif[k]$, which converges to $\unif[0,1]$ as $k\rightarrow\infty$. Since projection $(y,u)\mapsto u$ is continuous, the second marginal of $\rho$ is $\unif[0,1]$, hence $\rho\in\cC(A)$. So in order to prove the lemma we only need to show that $(\sample(\mu_k))_{k\geq 1}$ is a Chauchy sequence. 
		
		For all $k,m\in\bN$ and $n\geq\max\{k,m\}$ triangle inequality yields
		\begin{align*}
			\dW(\sample(\mu_k),&\sample(\mu_m))\\
			&\leq \dW(\sample(\mu_k),\sample(\RSS(\bbw_n,k)))\\
			&~~~~+\dW(\sample(\RSS(\bbw_n,k)),\sample(\RSS(\bbw_n,m)))\\
			&~~~~+\dW(\sample(\RSS(\bbw_n,m)),\sample(\mu_m)).
		\end{align*}
		Let $\varepsilon>0$. Since $\RSS(\bbw_n,j)\rightarrow \mu_j$ for each $j$ by assumption and since $\sample$ is continuous by Lemma~\ref{lemma:11} for each $j$ there exists an $N(j)\geq 1$ such that 
		$$\dW(\sample(\mu_j),\sample(\RSS(\bbw_n,j)))<\varepsilon/3~~\text{for all}~j\geq 1, n\geq N(j).$$
		By Lemma~\ref{lemma:22} (ii) the second term in the upper bound is bounded by a term that converges to zero as $\min\{k,m\}\rightarrow\infty$. Hence we can find $K\geq 1$ such that for all 
		$K\leq \min\{k,m\}\leq\max\{k,m\}\leq n$
		$$\dW(\sample(\RSS(\bbw_n,k)),\sample(\RSS(\bbw_n,m)))<\varepsilon/3$$
		holds. Now for all $\min\{k,m\}\geq K$ we obtain $\dW(\sample(\mu_k),\sample(\mu_m))<\varepsilon$ by using the above triangle inequality for some $n\geq \max\{N(k),N(m),\max\{k,m\}\}$, hence $(\sample(\mu_k))_{k\geq 1}$ is a Cauchy sequence.
	\end{proof}

	\begin{lemma}\label{lemma:44}
			Let $\rho\in\cC(A), k\geq 1$ and $U_1,\dots,U_k$ iid $\sim\unif[0,1]$ and $S_k:=\ps(U_1,\dots,U_k)$. Then
			\begin{equation*}
			\dW(\rho,\sample(\spread(\rho,k)))\leq \bE\Big(d_{\bR}\Big(U_1,\frac{S_k^{-1}(1)}{k}\Big)\Big),
			\end{equation*}
			The upper bound converges to zero as $k\rightarrow\infty$.
	\end{lemma}
	\begin{proof}
		Let $(Y_1,U_1),\dots,(Y_k,U_k)$ be iid $\sim\rho$, $S_k:=\ps(U_1,\dots,U_k)$. $\bbX:=(Y_1,U_1)$ has distribution $\rho$ and 
		$\bbX':=(Y_1,\frac{1}{k}S_k^{-1}(1))$ has distribution $\sample(\spread(\rho,k))$. The upper bound equals $\bE(d(\bbX,\bbX'))$ and converges to zero because $S_k^{-1}(1)$ converges almost surely to $U_1$ as $k\rightarrow\infty$ and $d_{\bR}$ is bounded.
	\end{proof}

	We are now ready to proof Theorem~2.
	
	\begin{proof}[Proof of Theorem \ref{thm:main2}]
		Let $(\bw_n)_{n\geq 1}$ be a sequence of words with $|\bw_n|\rightarrow\infty$.\\
		
		(ii)$\Rightarrow$(i): Suppose $\rho_{\bw_n}\rightarrow\rho\in\cC(A)$. We want to show that $\RSS(\bw_n,k)\rightarrow \spread(\rho,k)$ for each $k$. Triangle inequality yields
		\begin{align*}
		\dW_k(\RSS(\bw_n,k),&\spread(\rho,k))\\
		&\leq \dW_k((\RSS(\bw_n,k),\spread(\rho_{\bw_n},k))+\dW_k(\spread(\rho_{\bw_n},k),\spread(\rho,k))
		\end{align*}
		The first term converges to zero as $n\rightarrow\infty$ by Lemma~\ref{lemma:22} (i) and the second term converges to zero since $\spread(\cdot,k)$ is continuous by Lemma~\ref{lemma:11} (ii).\\
		
		(i)$\Rightarrow$(ii): Suppose $\RSS(\bw_n,k)\rightarrow\mu_k$ for each $k$. By Lemma~\ref{lemma:33} $\sample(\mu_k)$ converges towards some $\rho\in\cC(A)$ as $k\rightarrow\infty$. We want to show that $\rho_{\bw_n}\rightarrow \rho$. By triangle inequality we have for all $k\leq |\bw_n|$
		$$\dW(\rho_{\bw_n},\rho)\leq \dW(\rho_{\bw_n},\sample(\RSS(\bw_n,k)))+\dW(\sample(\RSS(\bw_n,k)),\rho).$$
		Since $\sample$ is continuous by Lemma~\ref{lemma:11} (i) and $\RSS(\bbw_n,k)\rightarrow\mu_k$, passing to the limes superior yields 
		$$\limsup\limits_{n\rightarrow\infty}\dW(\rho_{\bw_n},\rho)\leq \limsup\limits_{n\rightarrow\infty}\dW(\rho_{\bw_n},\sample(\RSS(\bw_n,k)))+\dW(\sample(\mu_k),\rho)$$
		for each $k$. The second term on the right side converges to zero as $k\rightarrow\infty$, hence for each $\varepsilon>0$ we can find $K\geq 1$ such that for all $k\geq K$
		$$\limsup\limits_{n\rightarrow\infty}\dW(\rho_{\bw_n},\rho)\leq \limsup\limits_{n\rightarrow\infty}\dW(\rho_{\bw_n},\sample(\RSS(\bw_n,k)))+\varepsilon.$$
		Since $\RSS(\bw,|\bw|)=\delta_{\bw}$ for each word $\bw$, it holds that 
		\begin{equation*}
			\rho_{\bw_n}=\sample(\RSS(\bw_n,|\bw_n|)),
		\end{equation*}
		hence for each $k\geq K$ we obtain by Lemma~\ref{lemma:22} (ii)
		$$\limsup\limits_{n\rightarrow\infty}\dW(\rho_{\bw_n},\rho)\leq \limsup\limits_{n\rightarrow\infty}\bE\Big(d_{\bR}\Big(\frac{S_k^{-1}(1)}{k},\frac{S_{|\bw_n|}^{-1}(1)}{|\bw_n|}\Big)\Big)+\varepsilon,$$
		where $\bbS=(S_j)_{j\geq 1}$ is a $S$-process. Since $n^{-1}S_n^{-1}(1)\rightarrow U_1$ a.s. and $|\bw_n|\rightarrow\infty$ we get for each $k\geq K$
		$$\limsup\limits_{n\rightarrow\infty}\dW(\rho_{\bw_n},\rho)\leq \bE\Big(d_{\bR}\Big(\frac{S_k^{-1}(1)}{k},U_1\Big)\Big)+\varepsilon.$$
		Now letting $k\rightarrow\infty$ yields $\limsup\limits_{n\rightarrow\infty}\dW(\rho_{\bw_n},\rho)\leq \varepsilon.$
		Since $\varepsilon>0$ was arbitrary, we get $\lim\limits_{n\rightarrow\infty}\dW(\rho_{\bw_n},\rho)=0.$\\

		We have shown that for each sequence $(\bbw_n)_{n\geq 1}$ with $|\bbw_n|\rightarrow\infty$ it holds that
		\begin{equation*}
			\rho_{\bbw_n}\rightarrow \rho~~~\Longrightarrow~~~\RSS(\bbw_n,k)\rightarrow \spread(\rho,k)~\text{for each $k$}
		\end{equation*}
		and 
		\begin{equation*}
			\RSS(\bbw_n,k)\rightarrow\mu_k~\text{for each $k$}~~~\Longrightarrow~~~\rho_{\bw_n}\rightarrow \lim_{k\rightarrow\infty}\sample(\mu_k).
		\end{equation*}
		
		We show that $\RSSB A$ and $\cC(A)$ are homeomorphic. For each $\rho\in\cC(A)$ we have $h(\rho):=(\spread(\rho,k))_{k\geq 1}\in\RSSB A$, see (\ref{eq:spreadrho}). Moreover, if 
		$\bmu=(\mu_k)_{k\geq 1}\in\RSSB A$ then $i(\bmu):=\lim_{k\rightarrow\infty}\sample(\mu_k)\in\cC(A)$ (see Lemma~\ref{lemma:33}). To prove Theorem~\ref{thm:main2} we need to show that
		$$h:\cC(A)\rightarrow \RSSB A,~~~\rho\mapsto (\spread(\rho,k))_{k\geq 1}$$
		is a homeomorphism and the inverse map is given by 
		$$i:\RSSB A\rightarrow \cC(A),~~~(\mu_k)_{k\geq 1}\mapsto \lim\limits_{k\rightarrow\infty}\sample(\mu_k).$$ 
		Let $\rho\in\cC(A)$. By Lemma \ref{lemma:44} we have that $\rho=\lim_{k\rightarrow\infty}\sample(\spread(\rho,k))$, hence $\rho=i(h(\rho))$. 
		Now let $\bmu\in\RSSB A$. By definition there is a sequence $(\bbw_n)_{n\geq 1}$ with $\RSS(\bbw_n,k)\rightarrow\mu_k$ for each $k$. We have shown above that $\rho_{\bw_n}$ converges towards 
		$\rho=i(\bmu)$ and that $\mu_k=\spread(\rho,k)$ for each $k$, hence $\bmu=h(i(\bmu))$, i.e. $h=i^{-1}$ and $i=h^{-1}$.\\
		
		The only thing left to show is that $h$ and $h^{-1}=i$ are continuous, i.e. we need to show that for all $\rho^n,\rho\in\cC(A)$
		\begin{equation*}
			\rho^n\rightarrow\rho~~~~\Longleftrightarrow~~~~\spread(\rho^n,k)\rightarrow\spread(\rho,k)~\text{for each}~k,
		\end{equation*}
		where '$\Longrightarrow$' follows directly from continuity of $\spread(\cdot,k)$, see Lemma~\ref{lemma:11} (i). We assume that the right hand side holds. Using triangle inequality yields for each $k$
		\begin{align*}
			\dW(\rho^n,\rho)\leq \dW(\rho^n,\sample&(\spread(\rho^n,k)))\\
												   &+\dW(\sample(\spread(\rho^n,k)),\sample(\spread(\rho,k)))\\
												   &+\dW(\sample(\spread(\rho,k)),\rho)
		\end{align*}
		Lemma~\ref{lemma:44} yields that the first and third term can be bounded by a term that converges to zero as $k\rightarrow\infty$. That is for each $\varepsilon>0$ we can find $K$ such that 
		for each $k\geq K$ it holds that 
		\begin{align*}
		\dW(\rho^n,\rho)\leq \varepsilon + \dW(\sample(\spread(\rho^n,k)),\sample(\spread(\rho,k))).
		\end{align*}
		Now $\spread(\rho^n,k)\rightarrow\spread(\rho,k)$ for each $k$ by assumption. Since $\sample$ is continuous by Lemma~\ref{lemma:11} we obtain $\limsup_{n\rightarrow\infty}\dW(\rho^n,\rho)\leq \varepsilon$ and hence $\varepsilon>0$ was arbitrary, $\rho^n\rightarrow\rho$.
	\end{proof}

	\section{Proof of Theorem 3}
	
	Before we proof that the backward filtration $\bbF$ generated by an ergodic GEWP is of product-type (Theorem~\ref{thm:main3}), we answer the question in what situations $\bbF$ is already generated by the eraser-process $\bbeta$. The answer was presented in Proposition~\ref{prop:filtrationgeneratedbyusual}: $\bbeta$ generates $\bbF$ iff $\law(\bbW,\bbeta)=\GL(\rho)$ and $\rho$ is of the form $\rho=\law(f(U),U)$ for some measurable $f:[0,1]\rightarrow A$ and $U\sim\unif[0,1]$. 
	\begin{proof}[Proof of Proposition~\ref{prop:filtrationgeneratedbyusual}]
		Let $(\bbW,\bbeta)$ be an ergodic GEWP with $\law(\bbW,\bbeta)=\GL(\rho), \rho\in\cC(A)$. Let $\bbU$ be the $U$-process corresponding to $\bbeta$ and let $Y_j:=W_{j,\eta_j}, j\geq 1$, hence 
		$\bbW_n=\ios(Y_1,\dots,Y_n,U_1,\dots,U_n)$ almost surely.\\
		Suppose $\rho=\law(f(U),U)$, i.e. there is some measurable $f:[0,1]\rightarrow A$ such that $Y_j=f(U_j)$ almost surely for all $j$. For each $k$ let $(U_{1:k},\dots,U_{k:k}):=\os(U_1,\dots,U_k)$ be the order statistics of $U_1,\dots,U_k$. In this case it holds that $\bbW_k=(f(U_{1:k}),\dots,f(U_{k:k}))$ almost surely and hence $\sigma(\bbW_k)\subas \sigma(U_{1:k},\dots,U_{k:k})$
		for each $k$. Proposition~\ref{prop:eta} yields $\sigma(U_{1:k},\dots,U_{k:k})\subas \sigma(\eta_n:n\geq k+1)$ and hence $\cF_k\subas \sigma(\eta_n:n\geq k+1)$, i.e. $\bbeta$ generates~$\bbF$.\\
		Now assume $\bbeta$ generates $\bbF$. We want to show that there is some measurable function $f$ with $Y_j=f(U_j)$ almost surely for all $j$. Since $(Y_j,U_j)_{j\geq 1}$ are iid it is enough to show that $Y_1=f(U_1)$ a.s. for some measurable $f$ which is true iff $\cL(Y_1|U_1)$ is almost surely a dirac measure. Since $\bbW_1=Y_1$ and $\sigma(\bbeta)\as\sigma(\bbU)$ by assumption we have that $Y_1$ is a.s. measurable with respect to $\bbU$, hence $\cL(Y_1|\bbU)=\cL(Y_1|U_1,U_2,\dots)$ is almost surely a dirac measure. Since $(Y_1,U_1)$ and $(U_2,U_3,\dots)$ are independent, $\law(Y_1|\bbU)=\law(Y_1|U_1)$ almost surely.
	\end{proof}
	In particular, Proposition~\ref{prop:filtrationgeneratedbyusual} yields that ergodic GEWPs of the form $\GL(\law(f(U),U))$ generate product-type filtrations. We now proof that this is true for every ergodic GEWP over a Borel spaces alphabet. As we have already explained, our proof proceeds as follows: First we prove it for finite alphabets, then for $A=[0,1]$ and then for general Borel spaces.
	
	\subsection{Case of Finite Alphabets}
	
	Let $A=\{1,2,\dots,m\}$ for some $m\geq 2$ and $(\bbW,\bbeta)$ be an ergodic GEWP over $A$. Our aim is to construct a process of local innovations $\bbeta^*$ that generates $\bbF$. The next lemma gives a general method to construct new processes of local innovations for any poly-adic-filtration:
	\begin{lemma}[see \cite{gael}, Lemma 2.1]\label{lemma:ll}
		Let $\bbF$ be a poly-adic filtrations and let $\bbeta$ be a process of local innovations with $\eta_n\sim\unif([n])$ for each $n$. For each $n$ let $\tau_n$ be a $\cF_n$-measurable $\bS_{n}$-valued random permutation. Then $\bbeta^*=(\eta^*_n)_{n\geq 1}$ defined by $\eta^*_n:=\tau_n(\eta_n)$ is a process of local innovations for $\bbF$.
	\end{lemma}
	\begin{proof}
		Let $\tau^{-1}_n$ be the inverse of $\tau_n$. Because $\sigma(\tau_n)\subseteq\cF_n$ it follows that $\sigma(\tau^{-1}_n)\subseteq\cF_n$. Because $\eta_n=\tau_n^{-1}(\eta^*_n)$ it follows that $\sigma(\eta_n)\subset\sigma(\eta^*_n)\vee\sigma(\tau^{-1}_n)\subseteq\sigma(\eta^*_n)\vee\cF_n$. Since $\cF_{n-1}\as\sigma(\eta_n)\vee\cF_{n}$ it follows that $\cF_{n-1}\subas\sigma(\eta^*_n)\vee\cF_n$. Now we need to show that $\eta^*_n$ is uniform on $[n]$ and independent from $\cF_n$: Let $B\in\cF_n$ and $j\in[n]$. It follows that 
		\begin{align*}
			\bP(\{\eta^*_n=j\}\cap B)&=\sum_{\pi\in\bS_{n}}\bP(\{\eta_n=\pi^{-1}(j)\}\cap B\cap\{\tau_n=\pi\})\\
								     &=\sum_{\pi\in\bS_{n}}\bP(\eta_n=\pi^{-1}(j))\cdot\bP(B\cap\{\tau_n=\pi\}),
		\end{align*}
		where the second equality holds because $B\cap\{\tau_n=\pi\}\in\cF_n$ and $\eta_n$ is independent from $\cF_n$. Because $\eta_n\sim\unif[n]$ it holds that $\bP(\eta_j=\pi^{-1}(j))=1/n$ for each $j$ and hence $\bP(\{\eta^*_n=j\}\cap B)=\frac{1}{n}\bP(B)$, so $\eta^*_n$ is uniform on $[n]$ and independent from $\cF_n$.
	\end{proof}
	
	We construct generating processes of local innovations based on the following 
	
	\begin{definition}\label{constr:umlabeln}
		Let $\bw=(w_1,\dots,w_n)\in A^n$. For each $i\in A$ let $n_i$ count the numbers of the letter $i$ in $\bw$, i.e. $n_i=\#\{1\leq k\leq n:w_k=i\}$. Let $\bw^*:=1^{n_1}2^{n_2}\dots m^{n_m}\in A^n$. The permutation $\pi_{\bw}\in\bS_n$ is defined by the requirements
		\begin{enumerate}
			\item[i)] $(\bw_{\pi_{\bw}^{-1}(1)},\dots,\bw_{\pi_{\bw}^{-1}(n)})=\bw^*$
			\item[ii)] $\pi_{\bw}^{-1}$ is increasing on each of the sets $\{1,\dots,n_1\},\{n_1+1,\dots,n_1+n_2\},\dots,\{n_1+\dots+n_{m-1},\dots,n\}$.
		\end{enumerate} 
		Figure \ref{fig:AP} shows an example.
	\end{definition}
	\begin{figure}[H]
		\includegraphics[scale=0.28]{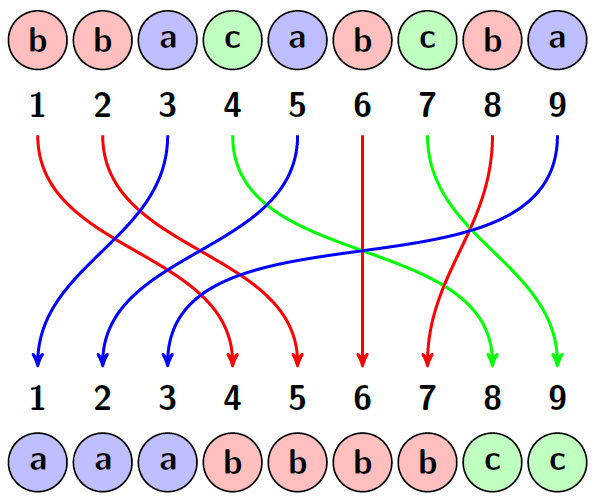}
		\caption{The alphabet $A=\{a,b,c\}=\{1,2,3\}$, above the word $\bw=(b,b,a,c,a,b,c,b,a)\in A^9$. An arrow $i\rightarrow j$ means $\pi_{\bw}(i)=j$.}\label{fig:AP}
	\end{figure}

	If $(\bbW,\bbeta)$ is a GEWP then $\bbeta^*$ with $\eta^*_n=\pi_{\bbW_n}(\eta_n)$ is a process of local innovations for $\bbF$ due to Lemma~\ref{lemma:ll}. 
	
	\begin{proposition}\label{prop:keyprop}
		If $(\bbW,\bbeta)$ is an ergodic GEWP over the finite alphabet $A$ then $\bbeta^*$ generates $\cF$ almost surely. In particular, the filtration $\bbF$ generated by any ergodic GEWP over a finite alphabet is of product-type.
	\end{proposition}

	We prove this proposition by an explicit construction of ergodic GEWPs involving $\rho$.

	\begin{definition}\label{def:constr}
		For $\rho\in\cC(A), i\in A=\{1,\dots,m\}$ let
		\begin{enumerate}
			\item $\alpha_i=\rho(\{i\}\times [0,1])$ and $\alpha=(\alpha_1,\dots,\alpha_m)$.
			\item $f:[0,1)\rightarrow A$ with $f(v)=i:\Leftrightarrow \alpha_1+\dots+\alpha_{i-1}\leq v<\alpha_1+\dots+\alpha_i$.
			\item $\rho_i:=\rho(\{i\}\times\cdot)$, i.e. $\rho_i$ is a measure on $[0,1]$. 
			\item $F_i^{-1}$ inverse quantile function of the measure $\rho_i$.
			\item $c:[0,1)\rightarrow[0,1]$ $c(v):=F^{-1}_{f(v)}(v-\alpha_1-\dots-\alpha_{f(v)-1})$.
		\end{enumerate}
	\end{definition}
	It is straightforward to check the following:	
	\begin{lemma}
		$V\sim\unif[0,1]\Longrightarrow (f(V),c(V))\sim\rho$.
	\end{lemma}
	
	\begin{lemma}\label{lemma:work}
		Let $\rho\in\cC(A)$ and $f, c$ be defined as in Definition~\ref{def:constr} and $\bbV=(V_j)_{j\geq 1}$ be a $U$-process. For each $n$ define 
		\begin{enumerate}
			\item $(Y_n,U_n):=(f(V_n),c(V_n))$.
			\item $S_n:=\ps(U_1,\dots,U_n)$ and $S_n^*:=\ps(V_1,\dots,V_n)$.
			\item $\eta_n:=S_n^{-1}(n)$ and $\eta^*_n:=(S^*_n)^{-1}(n)$.
			\item $\bbW_n:=\ios(Y_1,\dots,Y_n,U_1,\dots,U_n)$.
		\end{enumerate}
		Then for each $n$ it holds that
		\begin{enumerate}
			\item[(i)] $\eta^*_n=\pi_{\bbW_n}(\eta_n)$.
			\item[(ii)] $\bbW_n$ is a.s. measurable with respect to $\sigma(V_{1:n},\dots,V_{n:n})=\sigma(\os(V_1,\dots,V_n))$.
		\end{enumerate} 
	\end{lemma}
	
	\begin{proof}
		(i):~First we show that $\sigma:=(S_n^*)^{-1}\circ S_n$ is equal to $\pi_{\bbW_n}$. For this, we need to check that $\sigma$ has the two properties defining $\pi_{\bbW_n}$. Let $n_i:=\#\{k\in[n]:W_{n,k}=i\}$ be the number of occurrences of the letter $i$ in the word $\bbW_n$. By definition of $\bbW_n$ exactly $n_i$ of the values $V_1,\dots,V_n$ lie in the interval $[\alpha_1+\dots+\alpha_{i-1},\alpha_1+\dots+\alpha_i)$. Let $n_1+\dots+n_{i-1}+1\leq k\leq n_1+\dots+n_i$. Now $W_{n,\sigma^{-1}(k)}=Y_{S_n(\sigma^{-1}(k))}$ and $S_n(\sigma^{-1}(k))=S^*_n(k)$. Since $V_{S^*_n(k)}$ is the $k$-th largest among the $V$-values, $V_{k:n}$ has to lie in the interval $[\alpha_1+\dots+\alpha_{i-1},\alpha_1+\dots+\alpha_i)$ and so $W_{n,\sigma^{-1}(k)}=i$ by definition. Hence $\sigma$ shares property i) defining $\pi_{\bbW_n}$. The second property follows from the fact that the function $c$ is monotone increasing on intervals of the form $[\alpha_1+\dots+\alpha_{i-1},\alpha_1+\dots+\alpha_i)$. Hence $(S_n^*)^{-1}\circ S_n=\pi_{\bbW_n}$ which implies $\pi_{\bbW_n}(S_n^{-1}(n))=(S_n^*)^{-1}(n)$, so $\eta^*_n=\pi_{\bbW_{n}}(\eta_n)$ for every $n$ follows.\\
		(ii):~It is sufficient to prove that $\pi_{\bbW_n}$ is $\sigma(V_{1:n},\dots,V_{n:n})$-measurable, since $\bbW_n^*$ clearly is $\sigma(V_{1:n},\dots,V_{n:n})$-measurable and one can recover any word $\bw$ from $\bw^*$ and $\pi_{\bw}$. Let $\tau\in\bS_n$ be arbitrary. If $x_1,\dots,x_n$ are distinct real numbers, one has
		$$\ps(x_{\tau(1)},\dots,x_{\tau(n)})=\tau^{-1}\circ\ps(x_1,\dots,x_n).$$ 
		Let $\bbW^{\tau}_n$ be the random word that is constructed starting from $V_{\tau(1)},\dots,V_{\tau(n)}$ (instead from $V_1,\dots,V_n$) like explained in the lemma. Since $\pi_{\bbW_n}=(S^*_n)^{-1}\circ S_n$ one gets 
		$$\pi_{\bbW^{\tau}_n}=(\tau^{-1}\circ S^*_n)^{-1}\circ (\tau^{-1}\circ S_n)=\pi_{\bbW_n}.$$
		So $\pi_{\bbW_n}$ is measurable with respect to the exchangeable $\sigma$-field of $V_1,\dots,V_n$, which is equal to $\sigma(V_{1:n},\dots,V_{n:n})$.
	\end{proof}
	
	\begin{proof}[Proof of Proposition \ref{prop:keyprop}] 
		We start with the same definitions as in Lemma~\ref{lemma:work}. The process $(\bbW,\bbeta)=(\bbW_n,\eta_n)_{n\geq 1}$ is an ergodic GEWP with law $\GL(\rho)$. We want to show that the generated filtration $\bbF$ is of product-type. The process $\bbeta^*$ is the eraser-process corresponding to $\bbV$ and by (i) it is also a process of local innovations for $\bbF$. We want to show that $\bbeta^*$ generates $\bbF$. Since $\bbeta^*$ is a process of local innovations, we only need to check that $\sigma(\bbW_n)\subas\sigma(\bbeta^*_m:m\geq n+1)$ for each $n$. This follows immediately from the inclusion chain
		\begin{equation*}
			\sigma(\bbW_n)\subas \sigma(V_{1:n},\dots,V_{n:n})\subas \sigma(\eta^*_m:m\geq n+1),
		\end{equation*}
		where the first inclusion holds because of Lemma \ref{lemma:work} (ii) and the second because of Proposition \ref{prop:eta}.
	\end{proof}
	In Section~\ref{app:expl} we provide a graphical explanation for the fact that $\bbeta$ is not generating $\bbF$, but $\bbeta^*$ is in the case $A=\{0,1\}$ and $\rho=\unif\{0,1\}\otimes\unif[0,1]$.
	\begin{remark}
		Let $(\bbW,\bbeta)$ be an ergodic GEWP over the finite alphabet $A$ and let $(\bbY,L)$ be the representation as an exchangeable $A^{\bN}\times\bL$-valued pair. Define $\bbeta^*$ by $\eta^*_n=\pi_{\bbW_n}(\eta_n)$ and let $L^*$ be the exchangeable linear order corresponding to $\bbeta^*$. Since we have seen that $\bbeta^*$ is a generating process of local innovations for $\bbF$ and $\sigma(\bbeta^*)\as\sigma(L^*)$, we have that 
		$\sigma(\bbY,L,L^*)\as\sigma(L^*)$. The construction presented in Lemma~\ref{lemma:work} yields that the $A^{\bN}\times\bL\times\bL$-valued triple $(\bbY,L,L^*)$ is ergodic jointly exchangeable. In the sense of example 6.3 in \cite{glasner}, $(\bbY,L,L^*)$ is a special type of \emph{graph joining}, i.e. jointly exchangeable and $(Y,L)=f(L^*)$ for a suitable measurable function $f$. Note that this property is very special to our construction: not every process of local innovations $\bbeta^+$ yields an exchangeable triple $(\bbY,L,L^+)$ and not every exchangeable triple $(\bbY,L,L^+)$ yields a process $\bbeta^+$ that is a process of local innovations. 
	\end{remark}	
	\subsection{General Case}	
	We now lift our results concerning GEWPs over finite alphabets to the general case. We proceed analogously to \cite{laurent} and refer to that paper for more details. The arguments rely on the following facts:
	Let $f:A\rightarrow A'$ be a measurable function between Borel spaces and let $(\bbW,\bbeta)$ be a GEWP over $A$. Define the process $(f(\bbW),\bbeta)=(f(\bbW_n),\eta_n)_{n\geq 1}$ by $f(\bbW_n)=(f(W_{n,1}),\dots,f(W_{n,n}))$. 
	\begin{enumerate}
		\item[1.] $(f(\bbW),\bbeta)$ is a GEWP over $A'$. If $(\bbW,\bbeta)$ is ergodic, so is $(f(\bbW),\bbeta)$. This can be seen by noting that if $(Y_j,U_j)_{j\geq 1}$ is exchangeable/iid, so is $(f(Y_j),U_j)_{j\geq 1}$.
		\item[2.] Let $\bbF$ be the filtration generated by $(\bbW,\eta)$ and $\bbF^f$ be the filtration generated by $(f(\bbW),\bbeta)$. Then $\bbF^f$ is \emph{immersed} in $\bbF$. In our situation this means that the process $(f(\bbW_n),\bbeta_{n+1})_{n\geq 1}$ is markovian with respect to $\bbF$, i.e. 
		\begin{equation*}
		\law(f(\bbW_n),\bbeta_{n+1}|\cF_{n+1})=\law(f(\bbW_n),\bbeta_{n+1}|f(\bbW_{n+1}),\bbeta_{n+2})~~\text{a.s.},
		\end{equation*}
		which is true since both sides of the equation are a.s. equal to 
		\begin{equation*}
		\frac{1}{n+1}\sum_{j=1}^{n+1}\delta_{\big(\erase(f(\bbW_{n+1}),j),j\big)},
		\end{equation*}
		see \cite{laurent}, Section~3.
	\end{enumerate}
	
	\begin{lemma}\label{lemma:interval}
		An ergodic GEWP over the alphabet $A=[0,1]$ generates a product-type backward filtration.
	\end{lemma}
	\begin{proof}
		Let $(\bbW,\bbeta)$ be an ergodic GEWP over $A=[0,1]$ with backward filtration $\bbF$. For each $m$ let $f_m:[0,1]\rightarrow\{0,1,2,\dots,2^m\}$ be defined by $f_m(u):=2^{-m}\lceil 2^mu\rceil$. Let $\bbF^m=(\cF^m_n)_{n\geq 1}$ be the filtration generated by $(f_m(\bbW),\bbeta)$. By Proposition~\ref{prop:keyprop} every $\bbF^m$ is of product-type. It holds that $\cF_n^1\subseteq \cF_n^1\subseteq \cF_n^3\subseteq\cdots$, i.e. 
		for each $n$ the sequence $(\cF^m_n)_{m\geq 1}$ is a (forward) filtration. The Borel $\sigma$-field on $[0,1]$ is clearly generated by the maps $f_m, m\geq 1$, hence $\cF_n=\sigma(\cup_{m\geq 1}\cF^m_n)$ for each $n$. As we have explained above, each $\bbF^m$ is immersed in $\bbF$. Now there is a general theorem from filtration theory that yields that $\bbF$ is of product-type, we refer the reader to \cite{laurent}, Proposition 1 for details.
	\end{proof}
		
	The case of general Borel spaces follows easily:
	
	\begin{proof}[Proof of Theorem~\ref{thm:main3}]
		Let $(\bbW,\bbeta)$ be an ergodic GEWP over a Borel space $A$. By definition of Borel space there exists a Borel subset $E\subset [0,1]$ and a bijection $f:A\rightarrow E$ such that both $f$ and $f^{-1}$ are measurable. If $X$ is a $A$-valued RV then $\sigma(X)=\sigma(f(X))$. Now the process $(f(\bbW),\bbeta)$ is an ergodic GEWP over $[0,1]$ (letters are concentrated on the subset $E$) and the filtration $\bbF^f$ is of product-type by Lemma~\ref{lemma:interval}. Because $\sigma(\bbW_{n,i})=\sigma(f(\bbW_{n,i}))$ for each $n, i\in[n]$, it holds that $\bbF=\bbF^f$, hence $\bbF$ is of product-type.
	\end{proof}
	
	\section{Appendix}
	
	\subsection{Simplices}\label{app:simpl}
	
	We shortly explain how to see that $\cM(A)$ and $\cM'(A)$ are simplices using Theorem~9.1 in \cite{dynkin}. Let $S_1,S_2,\dots$ be a sequence of Borel spaces and for each $k\geq 1$ let $Q_k:S_{k+1}\rightarrow\cM_1(S_{k})$ be a measurable function. A stochastic process $\bbX=(X_k)_{k\geq 1}$ in which  
	$X_k$ takes values in $S_k$ is called $Q$-chain, if 
	\begin{equation*}
		\law(X_k|\sigma(X_m:m\geq k+1))=Q_k(X_{k+1})~~\text{almost surely for each }~k\geq 1.
	\end{equation*}
	Let $\cM$ be the set of all laws of $Q$-chains. Note that by inverting time, i.e. by considering $(\cdots,X_3,X_2,X_1)$, one sees that $\cM$ contains all laws of Markov-Chains over index $-\bN$ with given \emph{transition probabilities}. Theorem~9.1 in \cite{dynkin} can be applied and yields that $\cM$ is a simplex, the extreme points $\erg\cM\subseteq\cM$ are a measurable subset and a $Q$-chain is ergodic (law is extreme point) if and only if the terminal $\sigma$-field generated by the process is a.s. trivial (the terminal $\sigma$-field is $H$-sufficient for $\cM$ in the language of \cite{dynkin}). Note that $\sigma$-fields determining ergodicity by a.s. triviality are a.s. unique (see Theorem 3.2). $\cM(A)$ and $\cM'(A)$ fall into this set-up by considering
	\begin{enumerate}
		\item[1.] $S_k=A^k\times[k+1], Q_k((\bbw,i))=(k+1)^{-1}\sum_{j=1}^{k+1}\delta_{(\erase(\bbw,j),j)}$ and $X_k=(\bbW_k,\eta_{k+1})$.
		\item[2.] $S_k=A^k, Q_k(\bbw):=(k+1)^{-1}\sum_{j=1}^{k+1}\delta_{\erase(\bbw,j)}$ and $X_k=\bbW_k$.
	\end{enumerate}

	\subsection{Proof of Proposition 1}\label{app:prop}
		\begin{proof}
		(i) Let $X$ be a random variable with values in a polish space $S$ and let $(\cF_n)_{n\geq 1}$ be a backward filtration and $\cF_{\infty}:=\cap_n\cF_n$. By reverse martingale convergence for each bounded measurable function $f:S\rightarrow\bR$ it holds that $\bE(f(X)|\cF_n)\rightarrow\bE(f(X)|\cF_{\infty})$ almost surely. We refer the reader to the proof of Theorem 11.4.1, \cite{dudley} in which it is explained how one can now conclude that $\cL(X|\cF_n)$ converges almost surely weakly towards $\cL(X|\cF_{\infty})$. Considering $S=A^k$, $X=\bbW_k$, $\cF_n=\sigma(\bbW_m:m\geq n)$ and the fact that $\law(\bbW_k|\cF_n)=\RSS(\bbW_n,k)$ almost surely yields (i).\\
		(ii) This follows directly from (i).\\
		(iii) We have seen that a $\RSS$-chain $\bbW$ is ergodic iff the terminal $\sigma$-field $\cT:=\cap_n\sigma(\bbW_m:m\geq n)$ is a.s. trivial. Because of this (iii) follows from the more informative fact
		\begin{equation*}
			\cT\as\sigma\Big(\law(\bbW|\cT)\Big)\as\sigma\Big((\law(\bbW_k|\cT))_{k\geq 1}\Big).
		\end{equation*}
		The inclusions '$\supseteq$' are obvious. The first $\subas$ stems from the fact that $\cT$ is obtained as the terminal $\sigma$-field generated by $\bbW$ and the second $\subas$ stems from the fact 
		that $\bbW$ is a Markov~chain: $\law(\bbW|\cT)$ is determined by $\law(\bbW_1,\dots,\bbW_m|\cT), m\geq 1$ and $\law(\bbW_1,\dots,\bbW_m|\cT)$ is measurable with respect to $\cL(\bbW_m|\cT)$.\\
		(iv) For each $k\leq m$ we define a map $f_{m,k}:\cM_1(A^m)\rightarrow\cM_1(A^k)$ by
		\begin{equation*}
		f_{m,k}(\mu):=\int_{A^m}\RSS(\bv,k)d\mu(\bv).
		\end{equation*}
		In particular, it holds that $f_{m,k}(\RSS(\bbw,m))=\RSS(\bbw,k)$ for each $n\geq m\geq k$ and $\bbw\in A^n$. We show that $f_{m,k}$ is continuous, i.e. $\mu_n\rightarrow\mu$ implies $f_{m,k}(\mu_n)\rightarrow f_{m,k}(\mu)$: Let 
		$\bbJ=(J_1,\dots,J_k)$ be uniformly distributed on the set of all $1\leq j_1<\dots<j_k\leq m$. Now $\mu_n\rightarrow\mu$ implies $\mu_n\otimes\law(\bbJ)\rightarrow\mu\otimes\law(\bbJ)$. Since the map $(\bv,(j_1,\dots,j_k))\rightarrow(v_{j_1},\dots,v_{j_k})$ is continuous and $f_{m,k}(\mu)$ is the push-forward of $\mu\otimes\law(\bbJ)$ under this map, the continuity of $f_{m,k}$ follows from continuous mapping theorem. Now let $(\bbw_n)$ be a $\RSS$-convergent sequence with limit $\bmu$. We want to show that $\bmu$ satisfies (\ref{eq:maringallaws}), i.e. $f_{m,k}(\mu_m)=\mu_k$ for all $k\leq m$. This directly follows from continuity of $f_{m,k}$ and $f_{m,k}(\RSS(\bbw_n,m))=\RSS(\bbw_n,k)$ for each $m\geq k$ with $|\bbw_n|\geq m$.\\
		(v) By (i) $\bbW$ is almost surely $\RSS$-convergent towards $(\cL(\bbW_k|\cT))_{k\geq 1}$. Since $\bbW$ is assumed to be ergodic, $\cT$ is almost surely trivial by (iii) and hence $\cL(\bbW_k|\cT)=\cL(\bbW_k)$ for all~$k$.
	\end{proof}

	\subsection{A Graphical Explanation}\label{app:expl}

	We consider an ergodic GEWP $(\bbW,\bbeta)=(\bbW_n,\eta_n)_{n\geq 1}$ over the alphabet $A=\{0,1\}$ that has law $\GL(\rho)$ with $\rho=\unif\{0,1\}\otimes\unif[0,1]$, i.e. $\bbW_n$ is uniform over $\{0,1\}^n$ for each $n$. As we have explained in the end of Section~2.2, one can visualize the convergence of $\rho_{\bbW_n}$ towards $\rho$ as $n\rightarrow\infty$ by drawing the sets $\set(\bbW_n)$ and $\set(\rho)$, see (\ref{eq:setw}) and (\ref{eq:setrho}). We simulated the GEWP and draw the specific sets for some $n$, starting with $n=5$, the last picture is the limiting case, i.e. $\lim_n\set(\bbW_n)=\set(\rho)=\{(t/2,t/2):0\leq t\leq 1\}$:
	\begin{center}\includegraphics[scale=0.5]{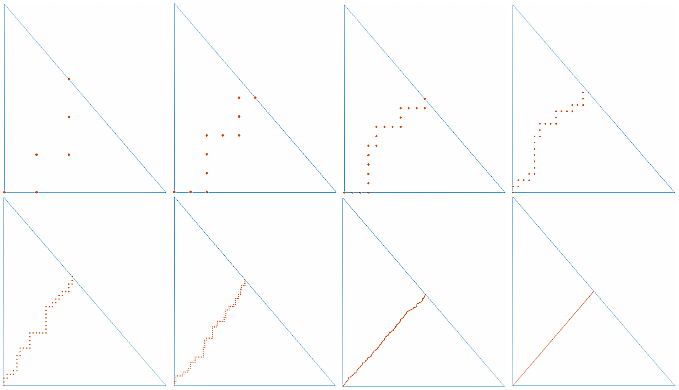}\end{center}
	Now let $\bbS=(S_n)_{n\geq 1}$ be the permutation process corresponding to $\bbeta$. We explain how to recover $\bbW_5$ from $\bbW_n$ and $S_n$ for finite $n\geq 5$: $\bbW_5$ is almost surely equal to the subsequence at positions $S_n^{-1}([5])$ in $\bbW_n$. In the pictures one can interpolate each $\set(\bbW_n)$ and draw line segments $(i/n,0)-(0,i/n)$ for all points $i\in S_n^{-1}([5])$. Each such line segment intersects the interpolated images of $\bbW_n$ and the way these intersection takes place (horizontal/vertically) encodes the information contained in $\bbW_5$. This is no longer true in the limiting case: If $\bbU$ is the $U$-process corresponding to $\bbS$, then 
	the line segments converge to $(U_i,0)-(0,U_i), i\in[5]$ and the images $\rho(\bbW_n)$ converge to the straight line segment $\set(\rho)$. The intersection behavior no longer contains any information, it is always the same angle. \emph{One can not recover $\bbW_5$ from the last picture}:
	\begin{center}\includegraphics[scale=0.5]{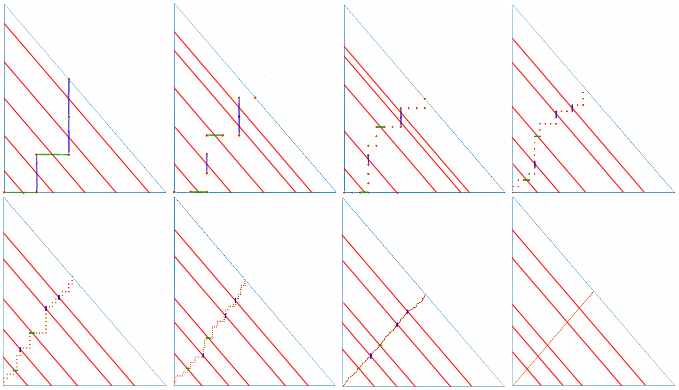}\end{center}
	Now consider $\bbeta^*$ as defined in Proposition~\ref{prop:keyprop} and let $\bbS^*, \bbV$ be the corresponding $S,U$-processes. One can visualize the information contained in $(S^*_n)^{-1}([5])$ by projecting the previously obtained intersection points either vertically or horizontally to the diagonal $(1,0)-(0,1)$, depending of the old intersection angle. Note that for finite $n$, the pictures below contain the same information as the pictures above. \emph{This is no longer the case as $n\rightarrow\infty$}: The last picture contains the information $V_i, i\in[5]$ (final points on $(1,0)-(0,1)$) and from that one can recover not just $U_i, i\in[5]$ \emph{but also $\bbW_5$}:
	\begin{center}\includegraphics[scale=0.5]{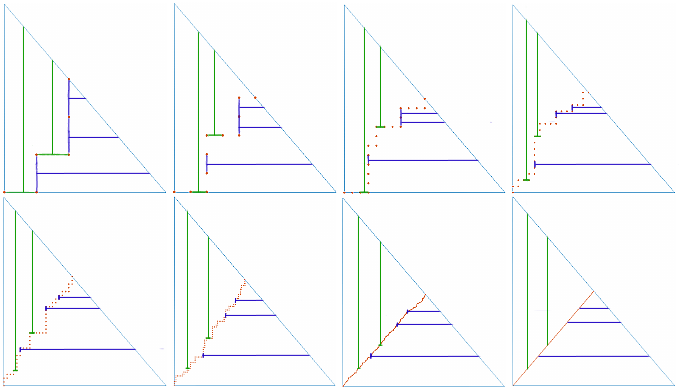}\end{center}
	This constructions works for any ergodic GEWP (the straight line just gets replaced with any other $\set(\rho), \rho\in\cC(\{0,1\})$). The pictures also help to understand Proposition~\ref{prop:filtrationgeneratedbyusual}: if $\rho=\law(f(U),U)$ for some measurable $f:[0,1]\rightarrow\{0,1\}$ and $U\sim\unif[0,1]$, then $\set(\rho)$ consists (basically) of vertical/horizontal parts and it is possible to recover $\bbW_5$ from intersection behavior with $(U_i,0)-(0,U_i)$ alone.

\end{document}